\documentclass[pdflatex,sn-mathphys-num]{sn-jnl}

\usepackage{graphicx}%
\usepackage{multirow}%
\usepackage{amsmath,amssymb,amsfonts}%
\usepackage{amsthm}%
\usepackage{mathrsfs}%
\usepackage[title]{appendix}%
\usepackage{xcolor}%
\usepackage{textcomp}%
\usepackage{manyfoot}%
\usepackage{booktabs}%
\usepackage{algorithm}%
\usepackage{algorithmicx}%
\usepackage{algpseudocode}%
\usepackage{listings}%

\usepackage{amssymb}
\usepackage{amsmath}
\usepackage{MnSymbol}

\usepackage[caption=false]{subfig}

\graphicspath{{./Figures/}}

\theoremstyle{thmstyleone}%
\newtheorem{theorem}{Theorem}
\theoremstyle{thmstyletwo}%

\theoremstyle{thmstylethree}%
\newtheorem{definition}{Definition}%

\begin{document}

\title[Minkowski Tensors of Planar Anisotropic Voronoi Diagrams]{Minkowski Tensors of Planar Anisotropic Voronoi Diagrams}


\author*[1]{\fnm{Nicolas} \sur{Venkovic}}\email{nicolas.venkovic@tum.de}


\author[1]{\fnm{Hartwig} \sur{Anzt}}\email{hartwig.anzt@tum.de}

\affil[1]{Chair of Computational Mathematics, School of CIT, TU Munich, Bildungscampus 2, Heilbronn, 74076, Germany}



\abstract{Anisotropic Voronoi diagrams, in which each cell is the region first reached by an ellipse growing from its site, generate non-convex cells whose integral geometry can be described using Minkowski tensors of arbitrary order. 
In this work, we derive semi-analytical formulas for Minkowski tensors of arbitrary order of the 0th, 1st and 2nd kind for planar anisotropic Voronoi cells that are star convex at their nucleation point. 
The resulting formulas involve integrals of contact functions, which describe the cell boundary through a diffeomorphic transformation. 
This transformation renders the geometry amenable to analytical treatment, enabling us to obtain expressions for the desired Minkowski tensors. 
Although most of the integrals lack closed-form solutions, they can be efficiently approximated numerically once the contact function and its derivatives are known along the boundary of the cell of interest.}


\maketitle
\section{Introduction}\label{sec:intro}
 
Tessellation models are commonly used to describe the cellular structure of materials, and to generate synthetic microstructures that are statistically equivalent to observed ones~\cite{Groeber2008,Jeulin2013,teferra2015tessellation}.
Calibrating such models, or inverting them to recover their generators from segmented images~\cite{Duan2014}, requires comparing descriptors of cell size, shape and orientation computed on the model with those measured on the data, for many cells and many candidate parameter sets.
Shape descriptors that are both accurate and inexpensive to evaluate are therefore a key ingredient of such workflows.
This requirement is especially challenging for microstructures with elongated grains, whose cells are anisotropic and bounded by curved interfaces.
 
A Voronoi diagram partitions an embedding space into cells, each cell being the set of points closer to a given generator, or site, than to any other site of the diagram.
Owing to this simple construction, Voronoi diagrams have become one of the fundamental data structures of computational geometry and a standard model of spatial tessellations in the natural and engineering sciences~\cite{Aurenhammer1991,okabe2009spatial}.
Many generalizations have been proposed, which differ by the nature of the sites, the embedding space and, most importantly for the present work, the distance function.
When the sites are points and the distance is Euclidean, the cells are convex polytopes.
As soon as the distance is allowed to depend on the site, however, the bisectors between neighboring sites become curved and the resulting cells are, in general, no longer convex~\cite{Boissonnat2006}.
 
Anisotropic Voronoi diagrams, introduced by Labelle and Shewchuk~\cite{labelle2003anisotropic}, equip each point site $\underline{x}_\alpha$ with its own symmetric positive-definite matrix $\mathbf{Z}_\alpha$ and measure distances with the corresponding site-dependent quadratic form.
They generalize multiplicatively weighted Voronoi diagrams, for which $\mathbf{Z}_\alpha$ is a scalar multiple of the identity.
Their bisectors are conic arcs, and their cells can be non-convex, and even disconnected~\cite{labelle2003anisotropic}.
An anisotropic Voronoi diagram admits an intuitive growth interpretation: the cell of $\underline{x}_\alpha$ is the region swept first by an ellipse centered at $\underline{x}_\alpha$, with shape and orientation prescribed by $\mathbf{Z}_\alpha$, when all ellipses start growing simultaneously and at the same rate.
This interpretation is at the root of their use as models of polycrystalline microstructures with elongated grains~\cite{teferra2015tessellation}, where they belong to a broader family of tessellations with anisotropic, curved cell boundaries that also includes generalized balanced power diagrams~\cite{Alpers2015,jung2023analytical}.
Anisotropic Voronoi diagrams also underpin anisotropic mesh generation~\cite{labelle2003anisotropic,Du2005}.
They should not be confused with Voronoi diagrams of ellipses in the Euclidean metric~\cite{emiris2006predicates,emiris2008predicates}, or with set Voronoi diagrams of aspherical particles~\cite{Schaller2013}, in which the sites themselves are extended objects.
In the context of microstructure generation, anisotropic Voronoi diagrams were resolved on a structured grid of voxels~\cite{teferra2015tessellation}, while exact algorithms for curved Voronoi diagrams rely on algebraic predicates for the bisectors~\cite{Boissonnat2006,emiris2006predicates,emiris2008predicates}, or on minimization diagrams obtained as projections of lower envelopes of distance functions~\cite{agarwal1996overlay,koltun2003partition,koltun2006overlays}.
More recently, an analytical representation of the cells of planar generalized balanced power diagrams was proposed~\cite{jung2023analytical}.
These approaches represent the cells, but do not address their morphological characterization, which has received much less attention.
 
Minkowski tensors provide a systematic and complete family of morphological descriptors.
They are tensor-valued generalizations of the Minkowski functionals (area, perimeter and Euler characteristic in the plane)~\cite{Mecke2000}, and they are, up to linear combinations, the only continuous, motion-covariant, tensor-valued valuations on convex bodies~\cite{McMullen1997,Alesker1999,Hug2008,Schneider2014}.
Their integral-geometric foundation makes them robust measures of anisotropy and orientation, and they have been used to characterize cellular, granular and porous structures, as well as spatial patterns in physics and materials science~\cite{Beisbart2002,Beisbart2006,Schroder2010,Kapfer2010,Schroder2011,Kapfer2012a,Kapfer2012b,Schroder2013,Mickel2013}.
For polygons and polyhedra, closed-form expressions are available and have been implemented efficiently, including for pixelated images~\cite{Mantz2008,Schroder2010,Schroder2013}.
Minkowski tensors can also be estimated from finite point samples~\cite{hug2017voronoi}, and density formulas are available for Boolean models~\cite{horrmann2014minkowski}.
All of these approaches, however, operate on a discretized or sampled representation of the set.
For sets bounded by curves, and in particular for cells with curvature discontinuities at their vertices, the cell is thus usually approximated by a polygon or a voxel set before its Minkowski tensors are evaluated.
This introduces discretization errors that are difficult to control, especially for tensors of high order and for tensors of the 2nd kind, for which the curvature of the approximating polygon is concentrated at its vertices.
 
In this work, we derive semi-analytical formulas for Minkowski tensors of arbitrary order of the 0th, 1st and 2nd kind for cells of planar anisotropic Voronoi diagrams that are star convex with respect to their site.
Our approach builds on a diffeomorphic transformation of the cell boundary into a cylinder $S^1\times(0,\infty)$, under which each visible segment of the boundary is described by an explicit contact function $\chi$ of a single angular variable.
Following the growth interpretation, the contact function gives the time at which two growing ellipses come into contact, and is thereby related to the distance of closest approach of two ellipses~\cite{Zheng2007}.
In terms of $\chi$ and its first two derivatives, we express the normal and tangent vectors, the length element and the curvature of the cell boundary in closed form.
Using these expressions, every Minkowski tensor is written as a finite sum of symmetric tensor products of the site position $\underline{x}_\alpha$ and of the principal directions $\underline{u}^\alpha_1$ and $\underline{u}^\alpha_2$ of $\mathbf{Z}_\alpha$, weighted by scalar one-dimensional integrals of the contact function.
The tensor structure is thus separated from the geometry: a single set of scalar integrals, which can be approximated by standard quadrature to arbitrary accuracy, serves all tensor orders, and no discretization of the cell is needed.
The contributions of the vertices of the cell boundary to the tensors of the 2nd kind are obtained in closed form through a parallel body construction~\cite{Schroder2010}.
 
The remainder of the paper is organized as follows.
Section~\ref{sec:avd} recalls the definition of anisotropic Voronoi diagrams.
Section~\ref{sec:diffeo} introduces the diffeomorphic transformation of cell boundaries and the contact function.
Section~\ref{sec:local} derives the local geometry of the cell boundary: normal and tangent vectors, surface and length elements, and curvature.
Section~\ref{sec:mt} presents the formulas for the Minkowski tensors of the 0th, 1st and 2nd kind.
Section~\ref{sec:comp} illustrates their computation on a model cell by means of polar glyphs of tensors up to order nine, and Section~\ref{conclusion} concludes the paper.

\section{Anisotropic planar Voronoi diagrams}\label{sec:avd}

Let $\mathcal{X}$ be a non-empty set equipped with a distance function $d:\mathcal{X}\times\mathcal{X}\rightarrow [0,\infty)$.
In its most general form, a Voronoi diagram is defined on $\mathcal{X}$ for any indexed collection $\{\mathcal{X}_\alpha\}_{\alpha\in I}$ of non-empty subsets $\mathcal{X}_\alpha$ of $\mathcal{X}$, referred to as sites, with index set $I\subseteq\mathbb{N}$.
The Voronoi cell $\mathcal{C}_\alpha\subseteq\mathcal{X}$ associated with a site $\mathcal{X}_\alpha$ is the set of elements of $\mathcal{X}$ whose distance to $\mathcal{X}_\alpha$ is not greater than their distance to any other site.
That is, upon letting the distance between a point $x\in\mathcal{X}$ and a subset $\mathcal{Y}\subseteq\mathcal{X}$ be given by $d(x,\mathcal{Y})=\inf\{d(x,y):\,y\in \mathcal{Y}\}$, the Voronoi cell associated with $\mathcal{X}_\alpha$ is given by
\begin{align}
\mathcal{C}_\alpha:=\{x\in \mathcal{X}:\,d(x,\mathcal{X}_\alpha)\leq d(x,\mathcal{X}_\gamma)\;\forall\,\gamma\in I\}.
\end{align}
Then, the Voronoi diagram is the indexed collection $\{\mathcal{C}_\alpha\}_{\alpha\in I}$ of Voronoi cells.
Under mild conditions on the sites and on the distance, these cells satisfy
\begin{align}
\operatorname{int}\mathcal{C}_\alpha\cap\operatorname{int}\mathcal{C}_\gamma=\emptyset
\;\forall\;\alpha,\gamma\in I\,\text{ s.t. }\,\alpha\neq \gamma
\;\text{ and }\;
\bigcup_{\alpha\in I}\mathcal{C}_\alpha=\mathcal{X},
\end{align}
where $\operatorname{int}$ denotes the interior of a set.
That is, the Voronoi diagram $\{\mathcal{C}_\alpha\}_{\alpha\in I}$ forms a tessellation of $\mathcal{X}$.
Such a general definition accounts for a broad variety of Voronoi diagrams.
Most significantly, Voronoi diagrams differ in their types of:
\begin{itemize}
\item Sites: while the most common practice is to let each site $\mathcal{X}_\alpha$ be a single element of $\mathcal{X}$, i.e., a point, the above definition allows for more general settings.
In particular, Voronoi diagrams have previously been defined where the sites were straight line segments~\cite{kirkpatrick1979efficient}, curved line segments~\cite{yap1987n}, discs and other convex areas~\cite{leven1987planning}, and more.
\item Embedding space: while the most common practice is to consider $\mathcal{X}$ as a Euclidean subspace of $\mathbb{R}^2$ or $\mathbb{R}^3$, generalized Voronoi diagrams have previously been defined on Riemannian manifolds~\cite{ehrlich1979dirichlet}, spheres and tori~\cite{brown1980,paschinger1982,yap1987n}, polyhedral surfaces~\cite{mount1985voronoi}, cones~\cite{dehne1987optimal}, and more.
\item Distance: while the great majority of Voronoi diagrams are based on Euclidean distances, many other distances have previously been used to define generalized Voronoi diagrams.
These include $p$-norms~\cite{lee1980two}, Hausdorff distances~\cite{aurenhammer1988voronoi}, convex distance functions~\cite{chew1985voronoi}, and more.
\end{itemize}
For a broad overview of generalized Voronoi diagrams, we refer the reader to chapter~3 in~\cite{okabe2009spatial}.
 
For the purpose of this work, as in most applications, we consider Voronoi diagrams of point sets.
Moreover, we restrict ourselves to planar Voronoi diagrams, that is, $\mathcal{X}\subseteq\mathbb{R}^2$.
In this context, we consider collections of sites of the form $\{\underline{x}_\alpha\}_{\alpha\in I}$, where $\underline{x}_\alpha\in\mathbb{R}^2$.
Of particular interest for the introduction of anisotropic Voronoi diagrams are multiplicatively weighted distances,
that is, distances of the form $d(\underline{x},\underline{x}_\alpha)=\|\underline{x}-\underline{x}_\alpha\|/w_{\alpha}$, where $w_\alpha$ is a site-dependent positive scalar, and $\|\cdot\|$ denotes the 2-norm.
Anisotropic Voronoi diagrams, which were introduced in~\cite{labelle2003anisotropic}, can be seen as generalizations of multiplicatively weighted Voronoi diagrams.
In $\mathbb{R}^2$, an anisotropic Voronoi diagram is defined by a marked point set $\{(\underline{x}_\alpha,\mathbf{Z}_\alpha)\}_{\alpha\in I}$, where each mark $\mathbf{Z}_\alpha\in\mathbb{R}^{2\times 2}$ is a symmetric positive-definite (SPD) matrix, leading to the distance defined by
\begin{align}
d(\underline{x},\underline{x}_\alpha)
:=
\sqrt{(\underline{x}-\underline{x}_\alpha)\cdot \mathbf{Z}_\alpha\cdot (\underline{x}-\underline{x}_\alpha)}
\;\forall\;\alpha\in I.
\end{align}
Strictly speaking, since it depends on the site, this distance is not a metric on $\mathbb{R}^2$ but a family of site-dependent distance functions, which is all that the definition of Voronoi cells requires.
Hereafter, we use the notation $d_\alpha(\underline{x}):=d(\underline{x},\underline{x}_\alpha)$.
Since $\mathbf{Z}_\alpha$ is SPD, there exist $v_1^\alpha\geq v_2^\alpha>0$ and orthonormal vectors $\underline{u}_1^\alpha,\underline{u}_2^\alpha\in\mathbb{R}^2$ such that
\begin{align}
\mathbf{Z}_\alpha
=
\frac{\underline{u}_1^\alpha\otimes\underline{u}_1^\alpha}{(v_1^{\alpha})^2}+
\frac{\underline{u}_2^\alpha\otimes\underline{u}_2^\alpha}{(v_2^{\alpha})^2},
\end{align}
where $\otimes$ denotes the tensor product.
Without loss of generality, we further assume $\underline{u}_1^\alpha\times\underline{u}_2^\alpha=\underline{e}_1\times\underline{e}_2$, i.e., $(\underline{u}_1^\alpha,\underline{u}_2^\alpha)$ is positively oriented, where $\underline{e}_1$ and $\underline{e}_2$ denote the standard basis vectors of $\mathbb{R}^2$.
Consequently, for any $t>0$, the level set 
\begin{align}
S_\alpha(t):=\{\underline{x}\in\mathbb{R}^2,\,d_\alpha(\underline{x})=t\}
\end{align}
is an ellipse centered at $\underline{x}_\alpha$, with semi-axes of lengths $t\,v_1^\alpha$ and $t\,v_2^\alpha$ along $\underline{u}_1^\alpha$ and $\underline{u}_2^\alpha$, respectively.
The cell $\mathcal{C}_\alpha$ is thus the region reached first by the ellipse growing from $\underline{x}_\alpha$ when all ellipses start growing simultaneously and at the same rate.

\section{Diffeomorphic transformation of anisotropic Voronoi cell boundaries}\label{sec:diffeo}
 
As we will see in the following sections, the boundary of a Voronoi cell carries most of the information about the morphology of the cell that is needed to evaluate Minkowski tensors.
Therefore, we first describe the boundary of an anisotropic Voronoi cell, and we then introduce star convexity as the property that allows us to parameterize the boundary of a non-convex cell explicitly.
 
\begin{definition}[Common curves and cell boundary]\label{def:boundary}
For every $\gamma\neq\alpha$ such that $\mathcal{C}_\alpha\cap\,\mathcal{C}_\gamma\neq\emptyset$, the common curve between the cells $\mathcal{C}_\alpha$ and $\mathcal{C}_\gamma$ is
\begin{align}
\mathcal{I}_{\alpha\gamma}:=\mathcal{C}_\alpha\cap\mathcal{C}_\gamma\subseteq\{\underline{x}\in\mathbb{R}^2,\,d_\alpha(\underline{x})=d_\gamma(\underline{x})\}.
\end{align}
The boundary of $\mathcal{C}_\alpha$ is the union $\mathcal{I}_\alpha:=\partial\mathcal{C}_\alpha=\bigcup_{\gamma}\mathcal{I}_{\alpha\gamma}$ of its common curves.
The points of $\mathcal{I}_\alpha$ that belong to more than one common curve are referred to as vertices, or common points of the cell.
\end{definition}
 
Each common curve $\mathcal{I}_{\alpha\gamma}$ is contained in the zero level set of the function
\begin{align}\label{eq:040}
f^{\alpha}_{\gamma}(\underline{y}):=d_\alpha(\underline{y})^2-d_\gamma(\underline{y})^2,
\end{align}
which is a conic.
The common curve is a regular curve provided that
\begin{equation}\label{eq:050}
\nabla f^{\alpha}_{\gamma}(\underline{y})
=2\left[\mathbf{Z}_\alpha\cdot(\underline{y}-\underline{x}_\alpha)-\mathbf{Z}_\gamma\cdot(\underline{y}-\underline{x}_\gamma)\right]
\neq\underline{0}
\;\;\forall\;\underline{y}\in\mathcal{I}_{\alpha\gamma}.
\end{equation}
Regularity is important, as it guarantees the existence of the local geometric quantities, such as normals and curvatures, used in this work to characterize the morphology of the cells.
The boundary $\mathcal{I}_\alpha$ is then a piecewise regular closed curve, whose normal and curvature are generally discontinuous at the common points.
 
\begin{definition}[Star-convex Voronoi cell]\label{def:star}
A Voronoi cell $\mathcal{C}_\alpha$ is said to be star convex at a point $\underline{y}\in\mathcal{C}_\alpha$ if, for every $\underline{x}\in\mathcal{C}_\alpha$, the line segment between $\underline{y}$ and $\underline{x}$ is contained in $\mathcal{C}_\alpha$.
That is, $\mathcal{C}_\alpha$ is star convex at $\underline{y}$ if and only if
\begin{align}
d_\alpha((1-s)\underline{y}+s\underline{x})\leq d_\gamma((1-s)\underline{y}+s\underline{x})
\;\;\forall\;
(\underline{x},\gamma,s)\in \mathcal{C}_\alpha\times I\times[0,1].
\end{align}
In what follows, we consider star convexity at the nucleation point $\underline{x}_\alpha$.
\end{definition}
 
The implicit representation of $\mathcal{I}_{\alpha\gamma}$ as part of the zero level set of Eq.~\eqref{eq:040} is the source of two difficulties.
First, it does not tell which part of the conic actually belongs to the cell boundary.
Second, it does not provide the coordinates of the points of $\mathcal{I}_{\alpha\gamma}$, so that computing integrals over the common curve, which is needed to evaluate Minkowski tensors, is not straightforward.
To circumvent these difficulties, we introduce the map
\begin{equation}\label{eq:020}
\begin{split}
\varphi_\alpha:S^1\times(0,\infty)&\,\rightarrow \mathbb{R}^2\setminus\{\underline{x}_\alpha\}\\
(\underline{s},t)&\,\mapsto\underline{x}_\alpha+t\,\mathbf{Z}_{\alpha}^{-1/2}\cdot\underline{s},
\end{split}
\end{equation}
where $S^1$ denotes the unit circle centered at the origin of $\mathbb{R}^2$.
Since $d_\alpha(\varphi_\alpha(\underline{s},t))=t$, the map $\underline{s}\mapsto\varphi_\alpha(\underline{s},t)$ traces the ellipse $S_\alpha(t)$, and the map $t\mapsto\varphi_\alpha(\underline{s},t)$ traces a ray emanating from $\underline{x}_\alpha$.
The map $\varphi_\alpha$ is a smooth bijection, whose inverse is given by
\begin{equation}\label{eq:030}
\begin{split}
\varphi_\alpha^{-1}:\mathbb{R}^2\setminus\{\underline{x}_\alpha\}&\,\rightarrow S^1\times(0,\infty)\\
\underline{y}&\,\mapsto
\left(
\frac{\mathbf{Z}_\alpha^{1/2}\cdot(\underline{y}-\underline{x}_{\alpha})}{\|\mathbf{Z}_\alpha^{1/2}\cdot(\underline{y}-\underline{x}_{\alpha})\|},\,
\|\mathbf{Z}_\alpha^{1/2}\cdot(\underline{y}-\underline{x}_{\alpha})\|
\right).
\end{split}
\end{equation}
Since $\varphi_\alpha^{-1}$ is also smooth, $\varphi_\alpha$ is a diffeomorphism.
Note that the second component of $\varphi_\alpha^{-1}(\underline{y})$ is $d_\alpha(\underline{y})$.
 
Rather than parameterizing $\mathcal{I}_{\alpha\gamma}$ directly, we study its image $\varphi^{-1}_{\alpha}(\mathcal{I}_{\alpha\gamma})$, which, since $\varphi_\alpha$ is a diffeomorphism, is a regular curve whenever $\mathcal{I}_{\alpha\gamma}$ is.
As shown in Theorem~\ref{theo:diffeo}, this image admits an explicit one-dimensional parameterization $\phi^{\alpha}_{\gamma}:(\theta_a,\theta_b)\rightarrow\varphi^{-1}_{\alpha}(\mathcal{I}_{\alpha\gamma})$ for segments of $\mathcal{I}_{\alpha\gamma}$ that are visible from $\underline{x}_\alpha$, from which an explicit representation $\varphi_{\alpha}\circ\phi^{\alpha}_{\gamma}:(\theta_a,\theta_b)\rightarrow\mathcal{I}_{\alpha\gamma}$ of the common curve follows.
This approach is summarized in Fig.~\ref{fig.1000}.
 
\begin{figure}
\centering
\includegraphics[clip,scale=.6]{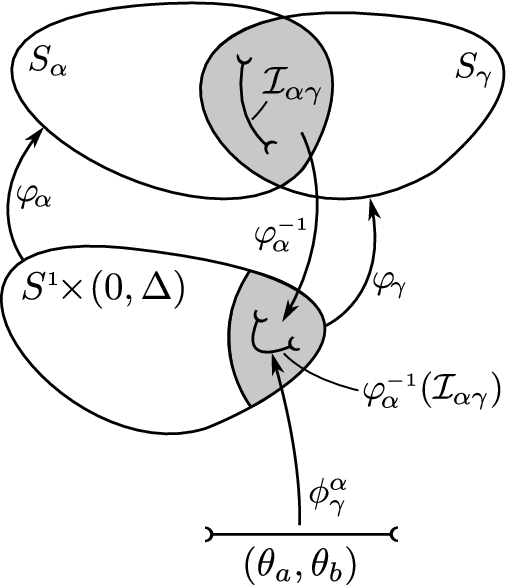}
\caption{Transformation and local parameterization of common curves $\mathcal{I}_{\alpha\gamma}$.}
\label{fig.1000}
\end{figure}
 
\begin{definition}[Visible boundary segment]\label{def:visible}
An open, connected subset $\delta\mathcal{I}\subseteq\mathcal{I}_{\alpha\gamma}$ that contains no vertex is said to be visible from $\underline{x}_\alpha$ if, for every $\underline{y}\in\delta\mathcal{I}$, the half-open segment $[\underline{x}_\alpha,\underline{y})$ is contained in the interior of $\mathcal{C}_\alpha$, and the ray from $\underline{x}_\alpha$ through $\underline{y}$ is not tangent to $\mathcal{I}_{\alpha\gamma}$ at $\underline{y}$.
\end{definition}
 
If $\mathcal{C}_\alpha$ is star convex at $\underline{x}_\alpha$ and no ray from $\underline{x}_\alpha$ is tangent to $\mathcal{I}_\alpha$, then every connected component of $\mathcal{I}_{\alpha\gamma}$ without its end points is visible from $\underline{x}_\alpha$.
 
\begin{theorem}[Diffeomorphic transformation of visible boundary segments]\label{theo:diffeo}
Let $\delta \mathcal{I}\subseteq \mathcal{I}_{\alpha\gamma}$ be a boundary segment of an anisotropic Voronoi cell $\mathcal{C}_\alpha$ that is visible from $\underline{x}_\alpha$, and let $\delta S^1:=\{\underline{s}\in S^1,\,\exists\,t>0,\;\varphi_\alpha(\underline{s},t)\in\delta\mathcal{I}\}$.
Then, $\delta S^1$ is an open arc of the unit circle, and $\varphi_\alpha^{-1}$ maps $\delta\mathcal{I}$ diffeomorphically onto the graph $\{(\underline{s},\chi_{\alpha\gamma}(\underline{s})):\,\underline{s}\in\delta S^1\}$ of the smooth contact function $\chi_{\alpha\gamma}:\delta S^1\rightarrow(0,\infty)$ given by
\begin{align}\label{eq:contact}
\chi_{\alpha\gamma}(\underline{s}):=
\frac{c}{b(\underline{s})+\sqrt{b(\underline{s})^2-a(\underline{s})\,c}}
\;\;\forall\;
\underline{s}\in\delta S^1,
\end{align}
where
\begin{align}
a(\underline{s}):=\underline{s}\cdot\mathbf{Z}_\gamma^\alpha\cdot\underline{s}-1,\quad
b(\underline{s}):=\underline{s}\cdot\mathbf{Z}_\gamma^\alpha\cdot\underline{x}_\gamma^\alpha,\quad
c:=\underline{x}_\gamma^\alpha\cdot\mathbf{Z}_\gamma^\alpha\cdot\underline{x}_\gamma^\alpha,
\end{align}
with $\underline{x}_\gamma^\alpha:=\mathbf{Z}_\alpha^{1/2}\cdot(\underline{x}_\gamma-\underline{x}_\alpha)$ and $\mathbf{Z}_\gamma^\alpha:=\mathbf{Z}_\alpha^{-1/2}\cdot\mathbf{Z}_\gamma\cdot\mathbf{Z}_\alpha^{-1/2}$.
In particular,
\begin{align*}
\delta \mathcal{I}=\{\varphi_\alpha(\underline{s},\chi_{\alpha\gamma}(\underline{s})):\,\underline{s}\in\delta S^1\}.
\end{align*}
\end{theorem}
 
\begin{proof}
Let $\underline{s}\in S^1$ and $t>0$, and let $\underline{y}=\varphi_\alpha(\underline{s},t)$.
Then $d_\alpha(\underline{y})=t$ and $\underline{y}-\underline{x}_\gamma=\mathbf{Z}_\alpha^{-1/2}\cdot(t\,\underline{s}-\underline{x}_\gamma^\alpha)$, so that
\begin{align}
d_\gamma(\underline{y})^2-d_\alpha(\underline{y})^2
=a(\underline{s})\,t^2-2\,b(\underline{s})\,t+c
=:g(\underline{s},t).
\end{align}
Since $\mathbf{Z}_\gamma^\alpha$ is SPD and $\underline{x}_\gamma\neq\underline{x}_\alpha$, we have $g(\underline{s},0)=c>0$.
By Definition~\ref{def:visible}, for every $\underline{y}\in\delta\mathcal{I}$ with $\varphi_\alpha^{-1}(\underline{y})=(\underline{s},\tau)$, $\tau$ is the smallest positive root of $g(\underline{s},\cdot)$, and this root is simple, since the ray is not tangent to $\mathcal{I}_{\alpha\gamma}$ at $\underline{y}$.
A simple root requires $b(\underline{s})^2-a(\underline{s})\,c>0$, and an elementary discussion of the signs of $a(\underline{s})$ and $b(\underline{s})$ shows that the smallest positive root is given by Eq.~\eqref{eq:contact}, so that $\tau=\chi_{\alpha\gamma}(\underline{s})$.
Moreover, since a ray meets $\delta\mathcal{I}$ only at the point where it first leaves $\mathcal{C}_\alpha$, the map $\delta\mathcal{I}\ni\underline{y}\mapsto\underline{s}$ is injective, and its image $\delta S^1$ is connected and open, i.e., an open arc.
Finally, $\partial_t g(\underline{s},\chi_{\alpha\gamma}(\underline{s}))=-2\sqrt{b(\underline{s})^2-a(\underline{s})\,c}\neq 0$, so that, by the implicit function theorem, $\chi_{\alpha\gamma}$ is smooth on $\delta S^1$.
Hence, $\varphi_\alpha^{-1}$, which is a diffeomorphism, maps $\delta\mathcal{I}$ onto the graph of the smooth function $\chi_{\alpha\gamma}$, which is diffeomorphic to $\delta S^1$.
\end{proof}
 
The contact function owes its name to the growth interpretation of anisotropic Voronoi diagrams, that is, $\chi_{\alpha\gamma}(\underline{s})$ is the time at which the ellipse growing from $\underline{x}_\alpha$, in the direction $\mathbf{Z}_\alpha^{-1/2}\cdot\underline{s}$, comes into contact with the ellipse growing from $\underline{x}_\gamma$.
Alternatively, an auxiliary problem formulation consists of evaluating the time at which the ellipse associated with the SPD matrix $\mathbf{Z}_\gamma^\alpha$ moving towards the origin comes into contact with the unit circle.
This auxiliary problem is illustrated in Fig.~\ref{fig.1005}.
If $\mathcal{C}_\alpha$ is star convex at $\underline{x}_\alpha$, each ray from $\underline{x}_\alpha$ leaves the cell exactly once, so that the contact function of the cell is
\begin{align}
\chi(\underline{s}):=\min_{\gamma}\chi_{\alpha\gamma}(\underline{s})\;\;\forall\;\underline{s}\in S^1,
\end{align}
where the minimum is taken over the neighbors $\gamma$ of $\mathcal{C}_\alpha$, and $\chi_{\alpha\gamma}(\underline{s}):=\infty$ whenever the ray in the direction $\underline{s}$ does not meet $\mathcal{I}_{\alpha\gamma}$.
The function $\chi$ is continuous on $S^1$, and smooth except at the directions $\underline{s}$ pointing to vertices of the cell.
Evaluating $\chi$ thus requires knowing the neighbors of $\mathcal{C}_\alpha$ and the arcs $\delta S^1$ over which each of them determines the cell boundary.
 
\begin{figure}
\centering
\includegraphics[clip,scale=.325]{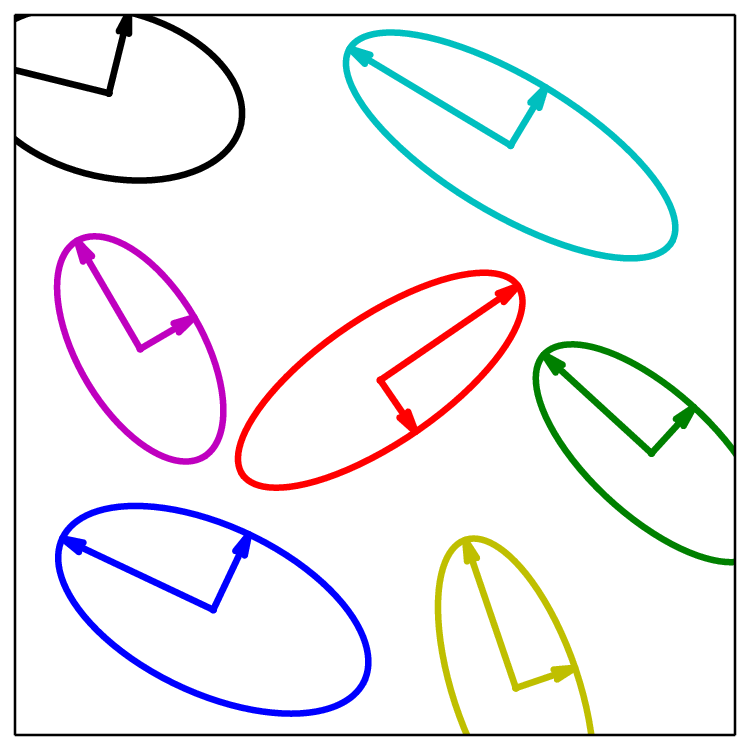}\includegraphics[clip,scale=.325]{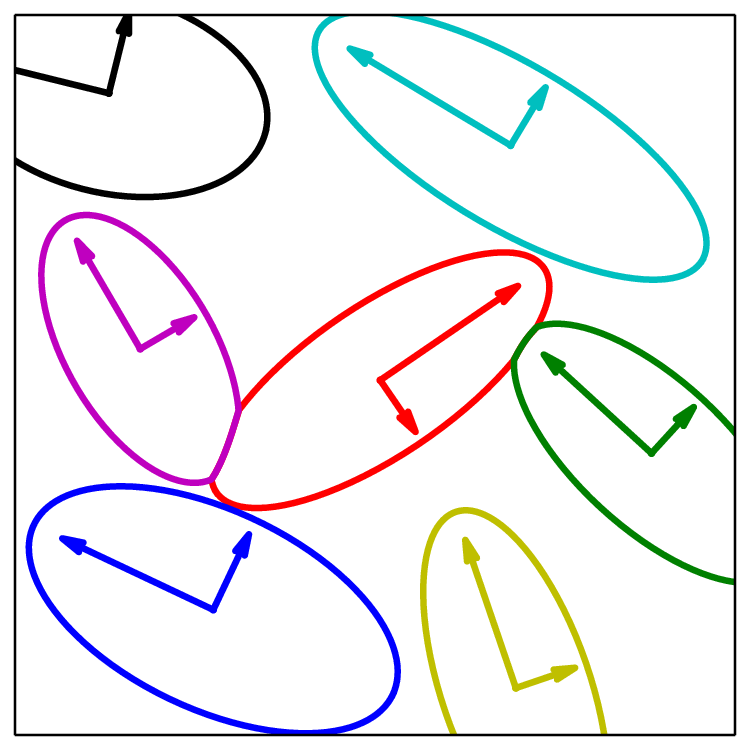}\includegraphics[clip,scale=.325]{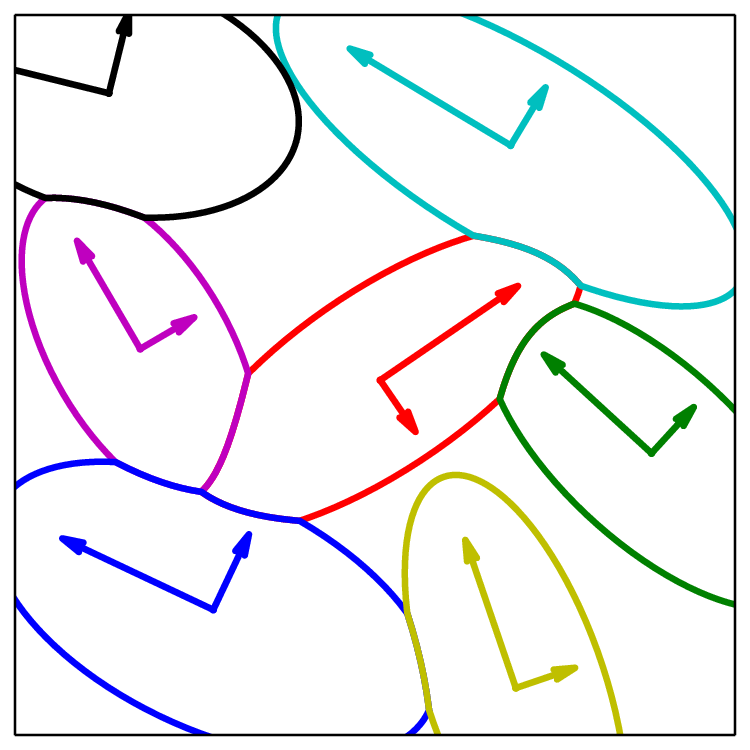}\\
\includegraphics[clip,scale=.3]{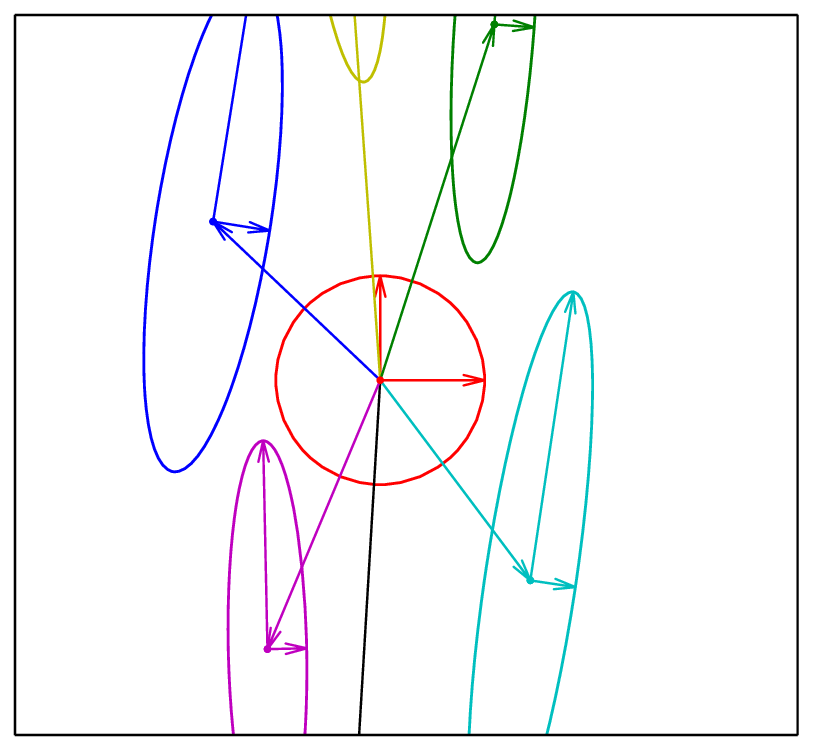}\includegraphics[clip,scale=.3]{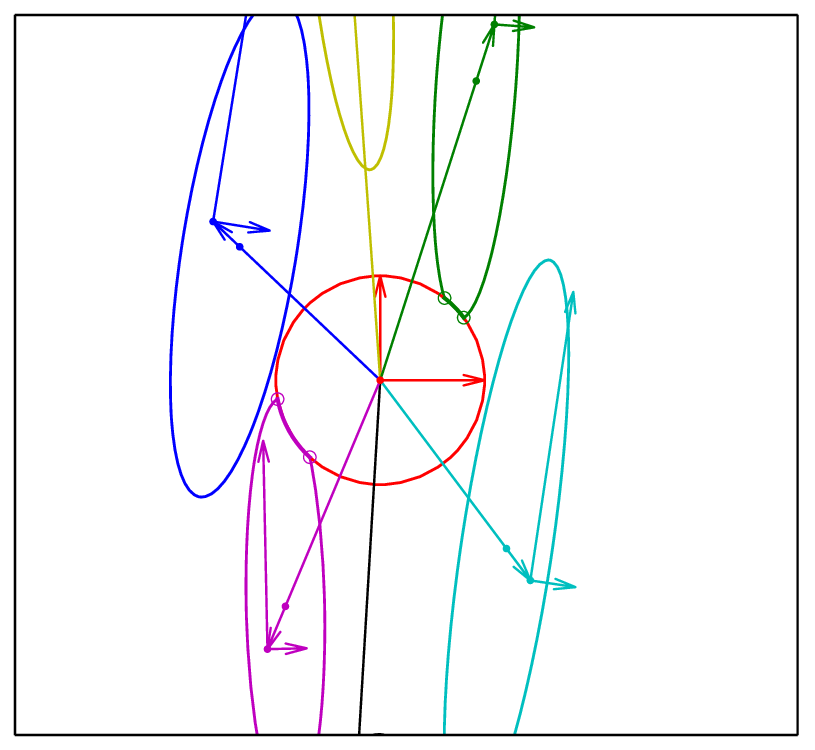}\includegraphics[clip,scale=.3]{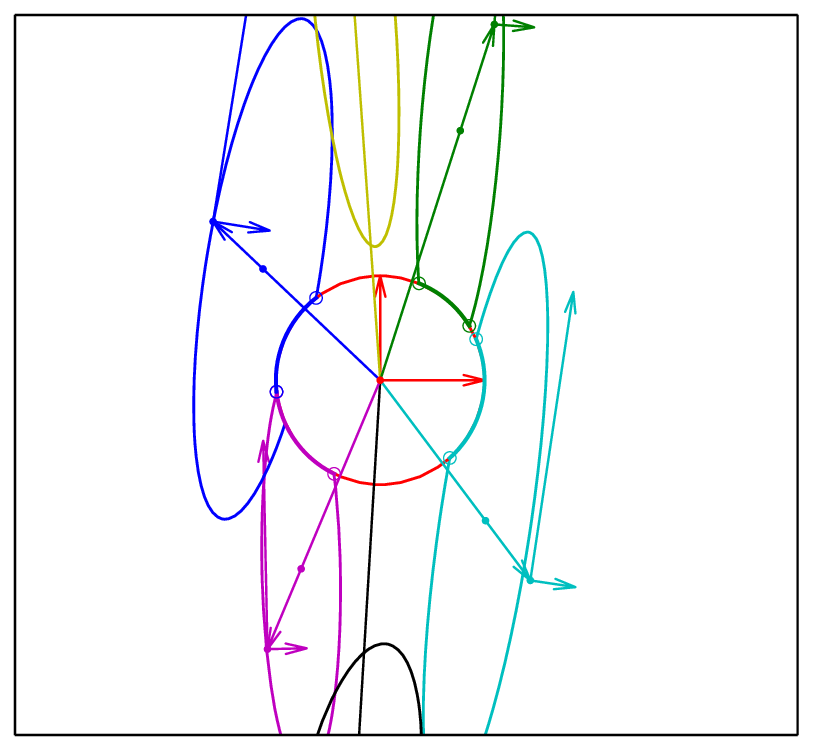}\\
\caption{Auxiliary problem associated with the diffeomorphic transformation of the cell boundary: points of the common curves are mapped onto points of the unit circle, at which the growing ellipses, transformed into fixed size ellipses moving towards the origin, come into contact at a time given by the contact function.}
\label{fig.1005}
\end{figure}
 
In the remainder of this work, we parameterize $S^1$ by an angle $\theta\in[-\pi,\pi)$ measured from $\underline{u}_1^\alpha$.
Let $\underline{x}^{\prime}(\theta):=\cos\theta\,\underline{e}_1+\sin\theta\,\underline{e}_2$ and $\mathbf{R}:=\underline{e}_1\otimes\underline{u}_1^{\alpha}+\underline{e}_2\otimes\underline{u}_2^{\alpha}$, so that $\mathbf{R}^T\cdot\underline{x}^{\prime}(\theta)=\cos\theta\,\underline{u}_1^\alpha+\sin\theta\,\underline{u}_2^\alpha\in S^1$.
With a slight abuse of notation, we write $\chi(\theta):=\chi(\mathbf{R}^T\cdot\underline{x}^{\prime}(\theta))$.
Then, every point of a cell $\mathcal{C}_{\alpha}$ that is star convex at $\underline{x}_{\alpha}$ is of the form $\underline{x}(\xi,\theta)$, with $\theta\in[-\pi,\pi)$ and $0\leq\xi\leq\chi(\theta)$, i.e.,
\begin{equation}
\mathcal{C}_{\alpha}=\left\{\underline{x}_{\alpha}+\xi\,\mathbf{Z}_{\alpha}^{-1/2}\cdot\mathbf{R}^T\cdot\underline{x}^{\prime}(\theta),\,0\leq\xi\leq\chi(\theta),\;\theta\in[-\pi,\pi)\right\}.
\end{equation}
Since $\mathbf{Z}_{\alpha}^{-1/2}=v_1^{\alpha}\,\underline{u}_1^{\alpha}\otimes\underline{u}_1^{\alpha}+v_2^{\alpha}\,\underline{u}_2^{\alpha}\otimes\underline{u}_2^{\alpha}$, this is recast as
\begin{equation}
\mathcal{C}_{\alpha}=\left\{\underline{x}_{\alpha}+\xi\,(v_1^{\alpha}\cos\theta\,\underline{u}_1^{\alpha}+v_2^{\alpha}\sin\theta\,\underline{u}_2^{\alpha}),\,0\leq\xi\leq\chi(\theta),\;\theta\in[-\pi,\pi)\right\}.
\end{equation}
Similarly, the boundary $\mathcal{I}_{\alpha}$ of $\mathcal{C}_{\alpha}$ is given by
\begin{equation}
\mathcal{I}_{\alpha}=\left\{\underline{x}_{\alpha}+\chi(\theta)\,\mathbf{Z}_{\alpha}^{-1/2}\cdot\mathbf{R}^T\cdot\underline{x}^{\prime}(\theta),\,\theta\in[-\pi,\pi)\right\},
\end{equation}
or, equivalently, by
\begin{equation}
\mathcal{I}_{\alpha}=\left\{\underline{x}_{\alpha}+\chi(\theta)\,(v_1^{\alpha}\cos\theta\,\underline{u}_1^{\alpha}+v_2^{\alpha}\sin\theta\,\underline{u}_2^{\alpha}),\,\theta\in[-\pi,\pi)\right\}.
\end{equation}

\section{Local geometry of planar anisotropic Voronoi cells}\label{sec:local}

The local geometry of planar anisotropic Voronoi cells is characterized in this section by expressing the normal and tangent vectors at the cell boundary, the infinitesimal surface and contour length elements as well as the curvature of the cell boundary in terms of the contact function $\chi$.

\subsection{Normal and tangent vectors of cell boundary}
Every point $\underline{x}\in\mathcal{I}_{\alpha}$ on the cell boundary $\mathcal{I}_\alpha$ of a star convex cell is parameterized in terms of some angle $\theta$ so that the normal $\underline{n}(\theta)$ is given by 
\begin{equation}
\underline{n}(\theta)=\frac{\partial\underline{t}(\theta)/\partial \theta}{\left\|\partial\underline{t}(\theta)/\partial \theta\right\|}
\end{equation}
where $\underline{t}(\theta)$ is the unit tangent vector to $\mathcal{I}_{\alpha}$ at $\underline{x}(\theta)$ given by 
\begin{equation}
\underline{t}(\theta)=\frac{\partial\underline{x}(\theta)/\partial \theta}{\left\|\partial\underline{x}(\theta)/\partial \theta\right\|}.
\end{equation}
From previous derivations, we have:
\begin{equation}\label{eq:150}
\underline{t}(\theta)=
\frac{A_\alpha(\theta)\underline{u}_1^{\alpha}+B_\alpha(\theta)\underline{u}_2^{\alpha}}{\sqrt{A_\alpha(\theta)^2+B_\alpha(\theta)^2}}.
\end{equation}
where
\begin{equation}
A_{\alpha}(\theta)\equiv v_1^{\alpha}\left[\chi^{\prime}(\theta)\cos\theta-\chi(\theta)\sin\theta\right]
\;\;
\mathrm{and}\;\;
B_{\alpha}(\theta)\equiv v_2^{\alpha}\left[\chi^{\prime}(\theta)\sin\theta+\chi(\theta)\cos\theta\right].
\end{equation}
Deriving an equation for the normal $\underline{n}(\theta)$ by further differentiation of the tangent yields complicated expressions with additional derivatives of the contact function.
Instead, since $\mathcal{I}_{\alpha}\in\mathbb{R}^2$, the normal can be obtained by a simple active rotation of $\underline{t}(\theta)$:
\begin{equation}
\underline{n}(\theta)=
\mathbf{Q}\cdot\underline{t}(\theta)
\end{equation}
where $\mathbf{Q}$ is the counter-clockwise rotation given as $\mathbf{Q}=\underline{e}_1\otimes\underline{e}_2-\underline{e}_2\otimes\underline{e}_1$ so that $\underline{n}(\theta)$ is an outward normal of the boundary of $\mathcal{C}_{\alpha}$.
Similarly, for the case in which $\underline{u}_1^{\alpha}\times\underline{u}_2^{\alpha}=\underline{e}_3$, we have
$\mathbf{Q}=\underline{u}_1^{\alpha}\otimes\underline{u}_2^{\alpha}-\underline{u}_2^{\alpha}\otimes\underline{u}_1^{\alpha}$ so that
\begin{equation}
\underline{n}(\theta)=
\frac{B_{\alpha}(\theta)\underline{u}_1^{\alpha}-A_{\alpha}(\theta)\underline{u}_2^{\alpha}}{\sqrt{A_\alpha(\theta)^2+B_\alpha(\theta)^2}}.
\end{equation}
Some tangent and normal vectors of a cell boundary are illustrated in Figure~\ref{fig.1100} for a model anisotropic Voronoi cell.
	
\begin{figure}
\centering
\includegraphics[clip,scale=.6]{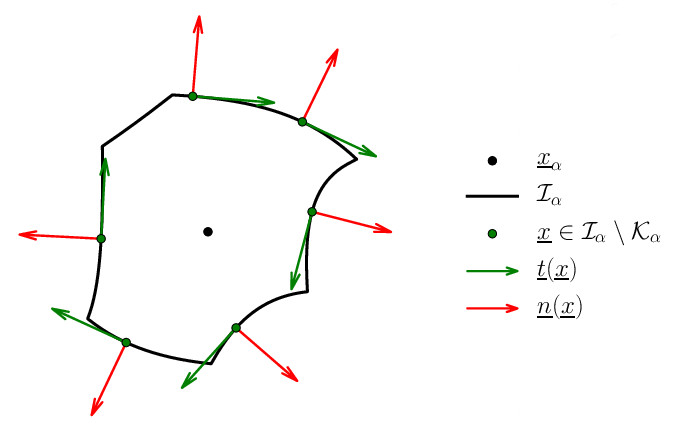}
\caption{Normal and tangent vectors of the cell boundary $\mathcal{I}_{\alpha}$.}
\label{fig.1100}
\end{figure}
	
\subsection{Parameterization of surface elements}
Integrals defined over $\mathcal{C}_{\alpha}$ can be recast as follows:
\begin{equation}
\int_{\mathcal{C}_{\alpha}}\bullet(\underline{x})\mathrm{d}A=
\int\limits_{\theta=-\pi}^{\pi}\int\limits_{\xi=0}^{\chi(\theta)}
\bullet(\underline{x}(\xi,\theta))
\left\|\frac{\partial \underline{x}}{\partial \xi}\times\frac{\partial \underline x}{\partial \theta}\right\|
\mathrm{d}\xi\mathrm{d}\theta
\end{equation}
where $\mathrm{d}A$ is an infinitesimal surface element of $\mathcal{C}_{\alpha}$ and $\bullet(\underline{x})$ is an arbitrary tensor field defined over $\mathcal{C}_{\alpha}$.
We have
\begin{equation}
\frac{\partial\underline{x}}{\partial \xi}=\mathbf{Z}_{\alpha}^{-1/2}\cdot\mathbf{R}^{T}\cdot\underline{x}^{\prime}=
\underline{u}_1^{\alpha}v_1^{\alpha}\cos\theta+\underline{u}_2^{\alpha}v_2^{\alpha}\sin\theta
\end{equation}
and
\begin{equation}
\frac{\partial\underline{x}}{\partial \theta}=\xi\;\mathbf{Z}_{\alpha}^{-1/2}\cdot\mathbf{R}^T\cdot\frac{\partial\underline{x}^{\prime}}{\partial \theta}=
\xi[\underline{u}_2^{\alpha}v_2^{\alpha}\cos\theta-\underline{u}_1^{\alpha}v_1^{\alpha}\sin\theta]
\end{equation}
so that
\begin{align}
\frac{\partial \underline x}{\partial \xi}\times\frac{\partial \underline x}{\partial \theta}
=&\,\left[\underline{u}_1^{\alpha}v_1^{\alpha}\cos\theta+\underline{u}_2^{\alpha}v_2^{\alpha}\sin\theta\right]\times
\xi[\underline{u}_2^{\alpha}v_2^{\alpha}\cos\theta-\underline{u}_1^{\alpha}v_1^{\alpha}\sin\theta]\\
=&\,\xi\left[v_1^{\alpha}v_2^{\alpha}\cos^2\theta+v_1^{\alpha}v_2^{\alpha}\sin^2\theta\right]\underline{e}_3\\
=&\,\pm\xi v_1^{\alpha}v_2^{\alpha}\underline{e}_3
\end{align}
where $\underline{u}_1^{\alpha}\times\underline{u}_2^{\alpha}=\pm\underline{e}_3$.
Consequently, we have
\begin{equation}
\left\|\frac{\partial \underline x}{\partial \xi}\times\frac{\partial \underline x}{\partial \theta}\right\|=\xi v_{1}^{\alpha}v_{2}^{\alpha}=\xi\;\mathrm{det}\mathbf{Z}^{-1/2}_{\alpha}
\end{equation}
and integrals on $\mathcal{C}_{\alpha}$ are given by:
\begin{align}
\int_{\mathcal{C}_{\alpha}}\bullet(\vec{x})\mathrm{d}A
=&\,v_1^{\alpha}v_2^{\alpha}\int\limits_{\theta=-\pi}^{\pi}\int\limits_{\xi=0}^{\chi(\theta)}
\bullet(\underline{x}(\xi,\theta))
\;\xi\;
\mathrm{d}\xi\mathrm{d}\theta\\
=&\,\frac{1}{\sqrt{\mathrm{det}\mathbf{Z}_{\alpha}}}\int\limits_{\theta=-\pi}^{\pi}\int\limits_{\xi=0}^{\chi(\theta)}
\bullet(\underline{x}(\xi,\theta))
\;\xi\;
\mathrm{d}\xi\mathrm{d}\theta.
\end{align}
The surface element $\mathrm{d}A$ at a point $\underline{x}(\xi,\theta)$ of the cell $\mathcal{C}_{\alpha}$ is given by $v_1^{\alpha}v_2^{\alpha}\xi\mathrm{d}\xi\mathrm{d}\theta$.
	
\subsection{Parameterization of contour length elements}
Every point $\underline{x}$ on the boundary $\mathcal{I}_{\alpha}$ of a planar Voronoi cell $\mathcal{C}_{\alpha}$ that is star convex at the nucleation point $\underline{x}_{\alpha}$ is identified as $\underline{x}(\theta)=\underline{x}(\xi=\chi(\theta),\theta)$ where $\theta\in[-\pi,\pi)$.
Therefore, integrals defined over $\mathcal{I}_{\alpha}$ can be recast as follows:
\begin{equation}
\int_{\mathcal{I}_{\alpha}}\bullet(\underline{x})\mathrm{d}\ell=
\int\limits_{\theta=-\pi}^{\pi}
\bullet(\underline{x}(\theta))
\left\|\frac{\partial \underline{x}}{\partial \theta}\right\|
\mathrm{d}\theta
\end{equation}
where $\mathrm{d}\ell$ is an infinitesimal length element of $\mathcal{I}_{\alpha}$ and $\bullet(\vec{x})$ is an arbitrary tensor field defined over $\mathcal{C}_{\alpha}\cup\mathcal{I}_{\alpha}$.
	
Given that for every $\underline{x}\in\mathcal{I}_{\alpha}$, $\underline{x}(\theta)=\underline{x}_{\alpha}+\chi(\theta)\mathbf{Z}_{\alpha}^{-1/2}\cdot\mathbf{R}^T\cdot\underline{x}^{\prime}(\theta)$, we have:
\begin{align}
\frac{\partial \underline{x}}{\partial \theta}
=&\,\frac{\partial \left[\underline{x}_{\alpha}+\chi(\theta)\mathbf{Z}_{\alpha}^{-1/2}\cdot\mathbf{R}\cdot\underline{x}^{\prime}(\theta)\right]}{\partial \theta}\\
=&\,\frac{\partial \chi}{\partial \theta}\mathbf{Z}_{\alpha}^{-1/2}\cdot\mathbf{R}^T\cdot\underline{x}^{\prime}(\theta)+
\chi(\theta)\mathbf{Z}_{\alpha}^{-1/2}\cdot\mathbf{R}^T\cdot\frac{\partial \underline{x}^{\prime}}{\partial \theta}.
\end{align}
We then define $\chi^{\prime}(\theta)\equiv\frac{\partial \chi (\theta)}{\partial \theta}$ so that
\begin{equation}
\begin{split}
\frac{\partial \underline{x}}{\partial \theta}=&\,
\chi^{\prime}(\theta)\mathbf{Z}_{\alpha}^{-1/2}\cdot\mathbf{R}^T\cdot\left[\underline{e}_1\cos\theta+\underline{e}_2\sin\theta\right]+
\chi(\theta)\mathbf{Z}_{\alpha}^{-1/2}\cdot\mathbf{R}^T\cdot\frac{\partial \left[\underline{e}_1\cos\theta+\underline{e}_2\sin\theta\right]}{\partial \theta}\\
=&\,
\chi^{\prime}(\theta)\left[\underline{u}_1^{\alpha}v_1^{\alpha}\cos\theta+\underline{u}_2^{\alpha}v_2^{\alpha}\sin\theta\right]+
\chi(\theta) \left[\underline{u}_2^{\alpha}v_2^{\alpha}\cos\theta-\underline{u}_1^{\alpha}v_1^{\alpha}\sin\theta\right]\\
=&\,
v_1^{\alpha}\left[\chi^{\prime}(\theta)\cos\theta-\chi(\theta)\sin\theta\right]\underline{u}_1^{\alpha}+
v_2^{\alpha}\left[\chi^{\prime}(\theta)\sin\theta+\chi(\theta)\cos\theta\right]\underline{u}_2^{\alpha}.
\end{split}
\end{equation}
Therefore, the infinitesimal contour length element $\mathrm{d}\ell$ of $\mathcal{I}_{\alpha}$ is given by
\begin{equation}
\mathrm{d}\ell=
\left\|\frac{\partial \underline{x}}{\partial \theta}\right\|
\mathrm{d}\theta=
\sqrt{A_\alpha(\theta)^2+B_\alpha(\theta)^2}\;\mathrm{d}\theta.
\end{equation}
	
\subsection{Curvature of cell boundary}
For a Voronoi cell $\mathcal{C}_{\alpha}$, the curvature $\kappa$ is discontinuous along $\mathcal{I}_{\alpha}$.
More precisely, curvature discontinuities arise at the common points of the structure, see Fig.~\ref{fig.1200}.
To address this issue, we define the following normals and tangents at a common point $\underline{x}\in\mathcal{K}_{\alpha}\subset\mathcal{I}_{\alpha}$:
\begin{equation}
\underline{n}^{-}(\underline{x})\equiv\lim_{\theta-\theta_{\!\underline{x}}\rightarrow 0^-}
\underline{n}(\underline{x}(\chi(\theta),\theta))
\;\;\;\mathrm{and}\;\;\;
\underline{n}^{+}(\underline{x})\equiv\lim_{\theta-\theta_{\!\underline{x}}\rightarrow 0^+}
\underline{n}(\underline{x}(\chi(\theta),\theta))
\;\;\forall\;\underline{x}\in\mathcal{K}_{\alpha}
\end{equation}
and
\begin{equation}
\underline{t}^{-}(\underline{x})\equiv\lim_{\theta-\theta_{\!\underline{x}}\rightarrow 0^-}
\underline{t}(\underline{x}(\chi(\theta),\theta))
\;\;\;\mathrm{and}\;\;\;
\underline{t}^{+}(\underline{x})\equiv\lim_{\theta-\theta_{\!\underline{x}}\rightarrow 0^+}
\underline{t}(\underline{x}(\chi(\theta),\theta))
\;\;\forall\;\underline{x}\in\mathcal{K}_{\alpha}
\end{equation}
where
\begin{equation}
\theta_{\!\underline{y}}\equiv\theta\;\mathrm{such}\;\mathrm{that}\;\underline{x}(\chi(\theta),\theta)=\underline{y}\in\mathcal{I}_{\alpha}.
\end{equation}
These quantities, which are depicted in Fig.~\ref{fig.1200} for the model Voronoi cell of Fig.~\ref{fig.1100}, will be used for the derivation of expressions of Minkowski tensors of type two. 

\begin{figure}
\centering
\includegraphics[clip,scale=.6]{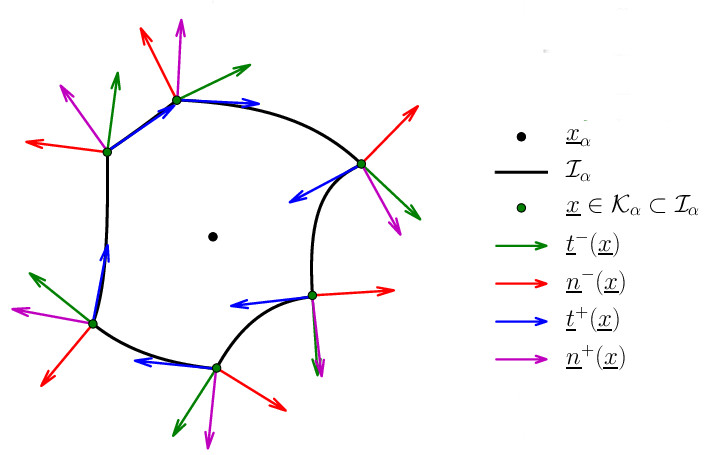}
\caption{Normal and tangent vectors at singular points of the cell boundary $\mathcal{I}_{\alpha}$.}
\label{fig.1200}
\end{figure}

At any point $\underline{x}(\chi(\theta),\theta)$ of the common curve $\mathcal{I}_{\alpha\beta}$, the local curvature is uniquely defined by
\begin{equation}\label{eq:200}
\kappa(\theta)=\frac{\|\partial \underline{t}(\theta)/\partial \theta\|}{\|\partial \underline{x}(\theta)/\partial \theta\|}.
\end{equation}
Upon minimal derivative work, we obtain the following expression:
\begin{equation}
\kappa(\theta)=
\sqrt{\frac{[A_{\alpha}(\theta)D_{\alpha}(\theta)-B_{\alpha}(\theta)C_{\alpha}(\theta)]^2}{[A_{\alpha}(\theta)^2+B_{\alpha}(\theta)^2]^3}}
\end{equation}
where
\begin{equation}
C_{\alpha}(\theta)\equiv
\frac{\partial A_{\alpha}(\theta)}{\partial \theta}=v_1^{\alpha}\left\{\vphantom{A^{A^A}}[\chi^{\prime\prime}(\theta)-\chi(\theta)]\cos\theta-2\chi^{\prime}(\theta)\sin\theta\right\}
\end{equation}
and
\begin{equation}
D_{\alpha}(\theta)\equiv
\frac{\partial B_{\alpha}(\theta)}{\partial \theta}=v_2^{\alpha}\left\{\vphantom{A^{A^A}}[\chi^{\prime\prime}(\theta)-\chi(\theta)]\sin\theta+2\chi^{\prime}(\theta)\cos\theta\right\}.
\end{equation}

\section{Minkowski tensors of star-convex cells}\label{sec:mt}
In this section, we provide analytical expressions for Minkowski tensors of 0th, 1st and 2nd kind of planar anisotropic Voronoi cells.
Those expressions are obtained by building on the descriptors of local geometry presented in the previous section.
For the sake of simplicity, we assume the Voronoi cell is simply connected, and star convex at its nucleation point.
	
\subsection{0th kind Minkowski tensors}
Minkowski tensors of the 0th kind are defined as follows for a simply connected anisotropic Voronoi cell $\mathcal{C}_{\alpha}$ that is star convex at its nucleation point:
\begin{equation}
\mathcal{W}^{r,0}_{0}(\mathcal{C}_{\alpha})\equiv\int_{ \mathcal{C}_{\alpha}}\underline{x}^{\otimes^r}\;\mathrm{d}A=
v_1^{\alpha}v_2^{\alpha}\int\limits_{\theta=-\pi}^{\pi}\int\limits_{\xi=0}^{\chi(\theta)}
\!\!\xi\;\underline{x}^{\otimes^r}
\mathrm{d}\xi\mathrm{d}\theta.
\end{equation}
Upon introducing the following integral:
\begin{equation}
I^{i,j}_0(\mathcal{C}_{\alpha})\equiv\binom{i+j}{i}\frac{(v_1^{\alpha})^{i+1}(v_2^{\alpha})^{j+1}}{i+j+2}\int\limits_{-\pi}^{\pi}[\chi(\theta)]^{i+j+2}(\cos\theta)^i(\sin\theta)^j\mathrm{d}\theta,
\end{equation}
we obtain the following expression:
\begin{equation}
\mathcal{W}^{r,0}_0(\mathcal{C}_{\alpha})=\sum_{i=0}^{r}\sum_{j=0}^{r-i}\binom{r}{i+j}\,
\underline{x}_{\alpha}^{\otimes^{r-i-j}}\!\!\!\odot\underline{u}_1^{{\alpha}^{\otimes^i}}\!\!\!\odot\underline{u}_2^{{\alpha}^{\otimes^j}}I^{i,j}_0(\mathcal{C}_{\alpha}).
\end{equation}
The presence of $\underline{x}_{\alpha}$ in the above expression is a clear indication that $\mathcal{W}^{r,0}_0$ is not translation invariant.
Indeed, the 0th kind Minkowski tensors of order greater than or equal to one (\textit{i.e.} $r\geq1$) are translation covariant.
That is, a translation of the cell $\mathcal{C}_{\alpha}$ by a vector $\underline{t}$ is characterized by the following transformation of the corresponding 0th kind Minkowski tensor:
\begin{equation}
\mathcal{W}^{r,0}_0(\mathcal{C}_{\alpha}\cupplus\underline{t})=\sum_{n=0}^r\binom{r}{n}\,\underline{t}^{\otimes^n}\!\!\odot\mathcal{W}^{r-n,0}_0(\mathcal{C}_{\alpha}).
\end{equation}
As a result, the anisotropy of a cell characterized by means of 0th kind Minkowski tensors is related to the origin with respect to which the points in $\mathcal{C}_{\alpha}$ are measured.
A common choice is to set the mass center $\underline{p}_0(\mathcal{C}_{\alpha})\equiv\mathcal{W}_0^{1,0}(\mathcal{C}_\alpha)/\mathcal{W}_0^{0,0}(\mathcal{C}_\alpha)$ as the origin and to characterize cell anisotropies with the following tensors:
\begin{equation}
{}^*\mathcal{W}^{r,0}_0(\mathcal{C}_{\alpha})\equiv
\mathcal{W}^{r,0}_0(\mathcal{C}_{\alpha}\cupplus-\underline{p}_0(\mathcal{C}_{\alpha}))=\sum_{n=0}^r\binom{r}{n}[-\underline{p}_0(\mathcal{C}_{\alpha})]^{\otimes^n}\odot\mathcal{W}^{r-n,0}_0(\mathcal{C}_{\alpha})\;\;\;\mathrm{for}\;r\geq 2
\end{equation}
Upon developing this expression, we see that the direct calculation of tensors ${}^*\mathcal{W}^{r,0}_0(\mathcal{C}_{\alpha})$ without resolving the mass center first does not simplify the calculation, nor does it decrease the number of integrals $I^{i,j}_0(\mathcal{C}_{\alpha})$ to be evaluated.
Therefore, instead, we recommend to proceed as follows.
First, compute the area covered by $\mathcal{C}_\alpha$ in $\mathbb{R}^2$:
\begin{equation}\label{eq:052}
W_0(\mathcal{C}_{\alpha})=I^{0,0}_0(\mathcal{C}_{\alpha}).
\end{equation}
Second, compute the mass center of $\mathcal{C}_\alpha$:
\begin{equation}
\underline{p}_0(\mathcal{C}_{\alpha})=\underline{x}_{\alpha}+\frac{\underline{u}_1^{\alpha}I^{1,0}_0(\mathcal{C}_{\alpha})+\underline{u}_2^{\alpha}I^{0,1}_0(\mathcal{C}_{\alpha})}{W_0(\mathcal{C}_{\alpha})}.
\end{equation}
Then, compute higher order tensors as follows:
\begin{equation}
{}^*\mathcal{W}^{r,0}_0(\mathcal{C}_{\alpha})=\sum_{i=0}^{r}\sum_{j=0}^{r-i}\binom{r}{i+j}\,
[\underline{x}_{\alpha}-\underline{p}_0(\mathcal{C}_{\alpha})]^{\otimes^{r-i-j}}\!\!\!\odot\underline{u}_1^{{\alpha}^{\otimes^i}}\!\!\!\odot\underline{u}_2^{{\alpha}^{\otimes^j}}I^{i,j}_0(\mathcal{C}_{\alpha})\,.
\end{equation}	

\subsection{1st kind Minkowski tensors}
The 1st kind Minkowski tensors are given as follows:
\begin{equation}
\mathcal{W}^{r,s}_1(\mathcal{C}_{\alpha})=\int_{\mathcal{I}_{\alpha}}\underline{x}^{\otimes^r}\!\!\odot\underline{n}(\underline{x})^{\otimes^s}\mathrm{d}\ell=
\int\limits_{\theta=-\pi}^{\pi}
\underline{x}^{\otimes^r}\!\!\odot\underline{n}(\theta)^{\otimes^s}
\sqrt{A_\alpha(\theta)^2+B_\alpha(\theta)^2}\;
\mathrm{d}\theta.
\end{equation}
Then, upon introducing the following integral:
\begin{equation}
I^{i,j,k,\ell}_1(\mathcal{C}_{\alpha})\equiv
(-1)^k
\binom{i+j}{j}
(v_1^{\alpha})^{i}(v_2^{\alpha})^{j}\!\!\!
\int\limits_{\theta=-\pi}^{\pi}
\frac{\chi(\theta)^{i+j}
(\cos\theta)^{i}(\sin\theta)^{j}
A_{\alpha}(\theta)^{k}
B_{\alpha}(\theta)^{\ell}}{\left\{A_{\alpha}(\theta)^2+B_{\alpha}(\theta)^2\right\}^{\frac{k+\ell-1}{2}}}
\mathrm{d}\theta
\end{equation}
we obtain the following expression for the 1st kind Minkowski tensors:
\begin{equation}
\mathcal{W}^{r,s}_1(\mathcal{C}_\alpha)=
\sum_{i=0}^r\sum_{j=0}^{r-i}\sum_{k=0}^{s}\binom{r}{i+j}\binom{s}{k}
\underline{x}_{\alpha}^{\otimes^{r-i-j}}\!\!\!\odot
\underline{u}_1^{\alpha^{\otimes^{s+i-k}}}\!\!\!\!\odot\underline{u}_2^{\alpha^{\otimes^{j+k}}}
I^{i,j,k,s-k}_1(\mathcal{C}_{\alpha}).	
\end{equation}
Special cases of the form $\mathcal{W}^{r,0}_1(\mathcal{C}_\alpha)$  are given by
\begin{equation}
\mathcal{W}^{r,0}_1(\mathcal{C}_\alpha)=
\sum_{i=0}^r\sum_{j=0}^{r-i}\binom{r}{i+j}
\underline{x}_{\alpha}^{\otimes^{r-i-j}}\!\!\!\odot
\underline{u}_1^{\alpha^{\otimes^{i}}}\!\!\!\!\odot\underline{u}_2^{\alpha^{\otimes^{j}}}
I^{i,j,0,0}_1(\mathcal{C}_{\alpha})
\end{equation}
which, similarly as $\mathcal{W}^{r,0}_0$ are translation covariant tensors for $r\geq 1$ and in which $I^{i,j,0,0}_1$ simplifies to
\begin{equation}
I^{i,j,0,0}_1(\mathcal{C}_{\alpha})\equiv
\binom{i+j}{j}
(v_1^{\alpha})^{i}(v_2^{\alpha})^{j}
\int\limits_{\theta=-\pi}^{\pi}
\chi(\theta)^{i+j}
(\cos\theta)^{i}(\sin\theta)^{j}
\sqrt{A_\alpha(\theta)^2+B_\alpha(\theta)^2}\;
\mathrm{d}\theta
\end{equation}
Other specific cases of the form $\mathcal{W}^{0,s}_1(\mathcal{C}_\alpha)$ are given by
\begin{equation}
\mathcal{W}^{0,s}_1(\mathcal{C}_\alpha)=
\sum_{k=0}^{s}\binom{s}{k}
\underline{u}_1^{\alpha^{\otimes^{s-k}}}\!\!\!\!\odot\underline{u}_2^{\alpha^{\otimes^{k}}}
I^{0,0,k,s-k}_1(\mathcal{C}_{\alpha})
\end{equation}
which are translation invariant and in which
\begin{equation}
I^{0,0,k,\ell}_1(\mathcal{C}_{\alpha})\equiv
(-1)^k
\int\limits_{\theta^{\prime}=-\pi}^{\pi}
\frac{A_\alpha(\theta)^{k}B_\alpha(\theta)^{\ell}}{\left[A_\alpha(\theta)^{2}+B_\alpha(\theta)^2\right]^{\frac{k+\ell-1}{2}}}
\mathrm{d}\theta.
\end{equation}

\subsection{2nd kind Minkowski tensors}
The 2nd kind Minkowski tensors of a planar set with continuous closed boundary $\mathcal{I}$ are given by
\begin{equation}
\begin{split}
\mathcal{W}^{r,s}_2\equiv\int_{\mathcal{I}}\kappa(\underline{x})\underline{x}^{\otimes^r}\!\!\odot\underline{n}(\underline{x})^{\otimes^s}\mathrm{d}\ell
\end{split}
\end{equation}
where $\kappa$ is the local curvature and $\underline{n}$ is the outward unit normal at a point $\underline{x}$ in $\mathcal{I}$.
As pointed out in the previous section, the curvature is not uniquely defined at the common points of the cell boundary $\mathcal{I}_{\alpha}$ of a Voronoi cell.
For a given Voronoi cell $\mathcal{C}_\alpha$, we denote the set of common points as $\mathcal{K}_{\alpha}$, and account for the effect of discontinuity in the curvature in an added term as follows:
\begin{equation}
\begin{split}
\mathcal{W}^{r,s}_2=\sum_{\beta}\int_{\mathcal{I}_{\alpha\beta}}\kappa(\underline{x})\underline{x}^{\otimes^r}\!\!\odot[\underline{n}(\underline{x})]^{\otimes^s}\mathrm{d}\ell+\sum_{\underline{x}\in\mathcal{K}_{\alpha}}\mathcal{D}^{r,s}(\underline{x})
\end{split}
\end{equation}
where every $\mathcal{I}_{\alpha\beta}$ is a disjoint subset of $\mathcal{I}_{\alpha}$ lying between two neighboring common points and such that $\mathcal{I}_{\alpha}=\cup_{\beta}\mathcal{\mathcal{I}_{\alpha\beta}}$.
The integrals of the first term of the right hand side can be recast in
\begin{equation}
\begin{split}
\int_{\mathcal{I}_{\alpha\beta}}\kappa(\underline{x})\underline{x}^{\otimes^r}\!\!\odot\underline{n}(\underline{x})^{\otimes^s}\mathrm{d}\ell=
\int\limits_{\theta=\theta^{\,-}_{\beta}}^{\theta^{\,+}_{\beta}}
\frac{\kappa(\theta)\,\underline{x}^{\otimes^r}\!\!\odot\underline{n}(\theta)^{\otimes^s}}{\left\{A_{\alpha}(\theta)^2+
B_{\alpha}(\theta)^2\right\}^{-1/2}}
\mathrm{d}\theta
\end{split}
\end{equation}
where $\mathcal{I}_{\alpha\beta}=\{\underline{x}(\chi(\theta),\theta)\,|\,\theta\in(\theta^{\,-}_{\beta},\theta^{\,+}_{\beta})\}$.
As we introduce the following integral:
\begin{equation}
I^{i,j,k,\ell}_2(\mathcal{I}_{\alpha\beta})\equiv
(-1)^k
\binom{i+j}{j}
(v_1^{\alpha})^{i}(v_2^{\alpha})^{j}\!\!\!
\int\limits_{\theta=\theta^{-}_\beta}^{\theta^{+}_\beta}
\frac{\chi(\theta)^{i+j}
(\cos\theta)^{i}(\sin\theta)^{j}
A_{\alpha}(\theta)^{k}
B_{\alpha}(\theta)^{\ell}}{\left[A_{\alpha}(\theta)^2+B_{\alpha}(\theta)^2\right]^{\frac{k+\ell-1}{2}}}
\,\kappa(\theta)\,
\mathrm{d}\theta,
\end{equation}
we obtain the following expression for the 2nd kind Minkowski tensors:
\begin{equation}
\mathcal{W}^{r,s}_2(\mathcal{C}_\alpha)=
\sum_{\beta}
\sum_{i=0}^r\sum_{j=0}^{r-i}\sum_{k=0}^{s}\binom{r}{i+j}\binom{s}{k}
\underline{x}_{\alpha}^{\otimes^{r-i-j}}\!\!\!\odot
\underline{u}_1^{\alpha^{\otimes^{s+i-k}}}\!\!\!\!\odot\underline{u}_2^{\alpha^{\otimes^{j+k}}}
I^{i,j,k,s-k}_2(\mathcal{I}_{\alpha\beta})
+\sum_{\underline{x}\in\mathcal{K}_{\alpha}}\mathcal{D}^{r,s}(\underline{x}).
\end{equation}
As is customary to account for discontinuities in the curvature, see, e.g., \cite{Schroder2010}, we resort to a parallel body construction for the derivation of an expression for the $\mathcal{D}^{r,s}(\underline{x})$ terms.
For that, we first let $B_{\epsilon}$ be a ball of radius $\epsilon$ in $\mathbb{R}^2$.
Then, we consider the set obtained by the Minkowski addition $\mathcal{C}_{\alpha}\cupplus B_{\epsilon}$ with continuous closed boundary $\partial(\mathcal{C}_{\alpha}\cupplus B_{\epsilon})$.
This construction is shown in Fig.~\ref{fig.1300} for the model Voronoi cell of this paper.
Then, an expression for $\mathcal{D}^{r,s}(\underline{x})$ is obtained by taking the limit of the integral similar to $\mathcal{W}_2^{r,s}$ over $\partial(\mathcal{C}_{\alpha}\cupplus B_{\epsilon})$ between the points $\underline{x}+\epsilon\underline{n}^{-}(\underline{x})$ and $\underline{x}+\epsilon\underline{n}^{+}(\underline{x})$:
\begin{equation}
\mathcal{D}^{r,s}(\underline{x})=
\lim_{\epsilon\rightarrow 0}
\int\limits_{\lambda=0}^{\gamma(\underline{x})}\kappa(\underline{x}^{*})\,\underline{x}^{{ *}^{\otimes^r}}\!\!\odot\underline{n}^{{ *}^{\otimes^s}}\!\epsilon\,\mathrm{d}\lambda=
\lim_{\epsilon\rightarrow 0}
\int\limits_{\lambda=0}^{\gamma(\underline{x})}\underline{x}^{{ *}^{\otimes^r}}\!\!\odot\underline{n}^{{ *}^{\otimes^s}}\mathrm{d}\lambda
\;\;\;\;\forall\underline{x}\in\mathcal{K}_{\alpha}
\end{equation}
where $\kappa(\underline{x}^{*})=1/\epsilon$ and the angle $\gamma(\underline{x})$ is such that
\begin{equation}
\cos(\gamma(\underline{x}))=\underline{n}^{-}(\underline{x})\cdot\underline{n}^{+}(\underline{x})
\;\;\;\;\forall\underline{x}\in\mathcal{K}_{\alpha}.
\end{equation}
The points $\underline{x}^{*}$ and normals $\underline{n}^{*}$ in the integral above can be parameterized as follows using the orthogonal basis formed by the pair $(\underline{n}^{-},\underline{t}^{-})$:
\begin{equation}
\underline{x}^{*}(\lambda\,|\,\underline{x})=\underline{x}+\epsilon\left[\underline{n}^{-}(\underline{x})\cos\lambda+\underline{t}^{-}(\underline{x})\sin\lambda\right],
\;\;\;\;\forall\lambda\in[0,\gamma(\underline{x})],\;\underline{x}\in\mathcal{K}_{\alpha}
\end{equation}
and
\begin{equation}
\underline{n}^{*}(\lambda\,|\,\underline{x})=\underline{n}^{-}(\underline{x})\cos\lambda+\underline{t}^{-}(\underline{x})\sin\lambda,
\;\;\;\;\forall\lambda\in[0,\gamma(\underline{x})],\;\underline{x}\in\mathcal{K}_{\alpha}.
\end{equation}
Thus for every $\underline{x}$ in $\mathcal{K}_{\alpha}$ we have
\begin{equation}
\begin{split}
\mathcal{D}^{r,s}(\underline{x})=&\,
\lim_{\epsilon\rightarrow 0}
\int\limits_{\lambda=0}^{\gamma(\underline{x})}\left\{\underline{x}+\epsilon\left[\underline{n}^{-}(\underline{x})\cos\lambda+\underline{t}^{-}(\underline{x})\sin\lambda\right]\right\}^{\otimes^r}\!\!\odot\left[\underline{n}^{-}(\underline{x})\cos\lambda+\underline{t}^{-}(\underline{x})\sin\lambda\right]^{\otimes^s}\mathrm{d}\lambda\\
=&\,
\sum_{k=0}^r\binom{r}{k}\lim_{\epsilon\rightarrow 0}
\int\limits_{\lambda=0}^{\gamma(\underline{x})}\epsilon^k\underline{x}^{\otimes^{r-k}}\odot\left[\underline{n}^{-}(\underline{x})\cos\lambda+\underline{t}^{-}(\underline{x})\sin\lambda\right]^{\otimes^k}\!\!\odot\left[\underline{n}^{-}(\underline{x})\cos\lambda+\underline{t}^{-}(\underline{x})\sin\lambda\right]^{\otimes^s}\mathrm{d}\lambda
\end{split}
\end{equation}
so that
\begin{equation}
\mathcal{D}^{r,s}(\underline{x})=
\int\limits_{\lambda=0}^{\gamma(\underline{x})}\underline{x}^{\otimes^{r}}\odot\left[\underline{n}^{-}(\underline{x})\cos\lambda+\underline{t}^{-}(\underline{x})\sin\lambda\right]^{\otimes^s}\mathrm{d}\lambda
\;\;\;\;\forall\;\underline{x}\in\mathcal{K}_{\alpha}
\end{equation}
which finally leads up to
\begin{equation}
\mathcal{D}^{r,s}(\underline{x})=
\sum_{k=0}^s\binom{s}{k}\,\underline{x}^{\otimes^{r}}\odot[\underline{n}^{-}(\underline{x})]^{\otimes{s-k}}\odot[\underline{t}^{-}(\underline{x})]^{\otimes{k}}
\mathcal{J}^{s-k,k}(\underline{x})
\;\;\;\;\forall\;\underline{x}\in\mathcal{K}_{\alpha}
\end{equation}
where
\begin{equation}
\mathcal{J}^{m,n}(\underline{x})\equiv\int\limits_{\lambda=0}^{\gamma(\underline{x})}(\cos\lambda)^m(\sin\lambda)^n\mathrm{d}\lambda.
\end{equation}

\begin{figure}
\centering
\includegraphics[clip,scale=.6]{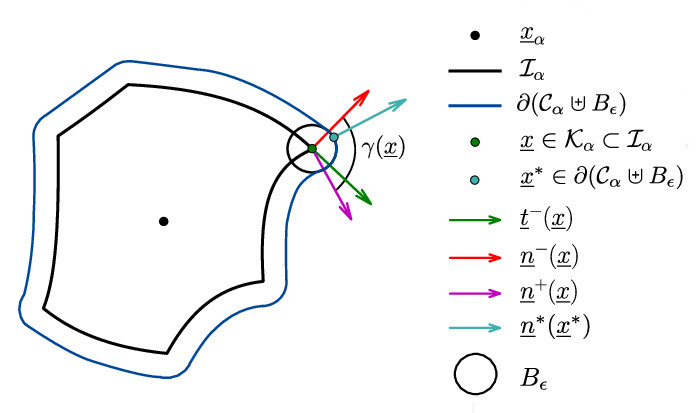}
\caption{Parallel body construction of the cell $\mathcal{C}_{\alpha}$.}
\label{fig.1300}
\end{figure}

Special cases of the form $\mathcal{W}^{r,0}_2(\mathcal{C}_\alpha)$  are given by
\begin{equation}
\mathcal{W}^{r,0}_2(\mathcal{C}_\alpha)=
\sum_{\beta}
\sum_{i=0}^r\sum_{j=0}^{r-i}\binom{r}{i+j}
\underline{x}_{\alpha}^{\otimes^{r-i-j}}\!\!\!\odot
\underline{u}_1^{\alpha^{\otimes^{i}}}\!\!\!\!\odot\underline{u}_2^{\alpha^{\otimes^{j}}}
I^{i,j,0,0}_2(\mathcal{I}_{\alpha\beta})
+\sum_{\underline{x}\in\mathcal{K}_\alpha}\mathcal{D}^{r,0}(\underline{x})
\end{equation}
which, similarly as $\mathcal{W}^{r,0}_0$ and $\mathcal{W}^{r,0}_1$, are translation covariant tensors for $r\geq 1$ and in which $I^{i,j,0,0}_2$ simplifies to
\begin{equation}
I^{i,j,0,0}_2(\mathcal{I}_{\alpha\beta})\equiv
\binom{i+j}{j}
(v_1^{\alpha})^{i}(v_2^{\alpha})^{j}
\int\limits_{\theta=\theta^{-}_\beta}^{\theta^{+}_\beta}
\frac{\chi(\theta)^{i+j}
(\cos\theta)^{i}(\sin\theta)^{j}}{\left[A_\alpha(\theta)^2+B_\alpha(\theta)^2\right]^{-\frac{1}{2}}}
\kappa(\theta)
\mathrm{d}\theta.
\end{equation}
The corresponding contributions from the common points in $\mathcal{I}_{\alpha}$ are 
\begin{equation}
\mathcal{D}^{r,0}(\underline{x})=\underline{x}^{\otimes^r}\gamma(\underline{x})
\;\;\;\;\forall\;\underline{x}\in\mathcal{K}_{\alpha}.
\end{equation}
Other specific cases of the form $\mathcal{W}^{0,s}_2(\mathcal{C}_\alpha)$ are given by
\begin{equation}
\mathcal{W}^{0,s}_2(\mathcal{C}_\alpha)=
\sum_{\beta}
\sum_{k=0}^{s}\binom{s}{k}
\underline{u}_1^{\alpha^{\otimes^{s-k}}}\!\!\!\!\odot\underline{u}_2^{\alpha^{\otimes^{k}}}
I^{0,0,k,s-k}_2(\mathcal{I}_{\alpha\beta})
+\sum_{\underline{x}\in\mathcal{K}_{\alpha}}\mathcal{D}^{0,s}(\underline{x})
\end{equation}
which are translation invariant and in which
\begin{equation}
I^{0,0,k,\ell}_2(\mathcal{I}_{\alpha\beta})\equiv
(-1)^k
(v_1^{\alpha})^{k}(v_2^{\alpha})^{\ell}
\int\limits_{\theta=\theta^{-}_\beta}^{\theta^{+}_\beta}
\frac{\kappa(\theta)}{\left[A_\alpha(\theta)^2+B_\alpha(\theta)^2\right]^{\frac{k+\ell-1}{2}}}
\mathrm{d}\theta.
\end{equation}
The corresponding contributions from the singular points in $\mathcal{I}_{\alpha}$ are 
\begin{equation}
\mathcal{D}^{0,s}(\underline{x})=
\sum_{k=0}^s\binom{s}{k}\,[\underline{n}^{-}(\underline{x})]^{\otimes{s-k}}\odot[\underline{t}^{-}(\underline{x})]^{\otimes{k}}
\mathcal{J}^{s-k,k}(\underline{x})
\;\;\;\;\forall\;\underline{x}\in\mathcal{K}_{\alpha}.
\end{equation}

\section{Validation, computation and visualization}\label{sec:comp}
 
\subsection{Validation of the cell boundary representation}\label{sec:validation}
 
All the formulas of Section~\ref{sec:mt} rely on the representation of the cell boundary by the contact function $\chi$.
We therefore verified this representation before evaluating any Minkowski tensor.
 
First, we verified the closed-form contact function of Eq.~\eqref{eq:contact} against a direct numerical search for the first contact between two growing ellipses.
For this purpose, we drew sites $\underline{x}_\alpha,\underline{x}_\gamma\in\mathbb{R}^2$, SPD matrices $\mathbf{Z}_\alpha,\mathbf{Z}_\gamma$, and directions $\underline{s}\in S^1$ at random.
For each of the 1115 configurations in which the ray $t\mapsto\varphi_\alpha(\underline{s},t)$ meets the bisector, we located the first sign change of $d_\gamma^2-d_\alpha^2$ along the ray by dense sampling, and compared its position with $\chi_{\alpha\gamma}(\underline{s})$.
In all cases, both values agreed to within the sampling resolution, i.e., to a relative accuracy of $10^{-3}$.
 
Second, we verified that the cell boundaries resolved by the contact function coincide with those of digitized anisotropic Voronoi tessellations, obtained by assigning every pixel of a regular grid to the site $\underline{x}_\alpha$ that minimizes $d_\alpha$ at the pixel center, in the spirit of the voxel-based construction of~\cite{teferra2015tessellation}.
The cell boundaries obtained from the contact function, superimposed on the digitized cells, coincide with the interfaces between pixels assigned to different sites.
This confirms both the identification of the neighbors of each cell and the angular ranges over which each neighbor determines the cell boundary.
 
Finally, we verified the area formula of Eq.~\eqref{eq:052} on a non-convex test cell with 17 neighbors.
The area obtained from $\mathcal{W}_0^{0,0}(\mathcal{C}_\alpha)=I_0^{0,0}(\mathcal{C}_\alpha)$, evaluated with $2\times10^5$ equally spaced values of $\theta$, agrees with the area of the polygon formed by the corresponding boundary points to a relative accuracy of $2\times10^{-6}$.
 
\subsection{Computation and visualization of Minkowski tensors}
 
The number of independent Cartesian components of a symmetric tensor of order $r$ defined in $\mathbb{R}^d$ is given by~\cite{Comon2006}
\begin{equation}
n_c=\binom{r+d-1}{r},
\end{equation}
so that the Minkowski tensor $\mathcal{W}^{r,s}_{\nu}(\mathcal{C}_\alpha)$ of a Voronoi cell $\mathcal{C}_{\alpha}\subset\mathbb{R}^2$, which is of order $r+s$, possesses $r+s+1$ independent components.
We compute these components for a series of Minkowski tensors of the model Voronoi cell used throughout the paper.

In order to visualize these tensors, we define the polar glyph of a symmetric tensor $\mathcal{A}$ of order $r$ as the curve
\begin{equation}
\mathcal{P}(\mathcal{A}):=\left\{\mathcal{A}[\underline{x}^{\otimes r}]\;\underline{x}:\;\underline{x}\in S^1\right\},
\end{equation}
where $\mathcal{A}[\underline{x}^{\otimes r}]:=\underline{x}\cdot[\cdots[\underline{x}\cdot\mathcal{A}]\cdots]$ denotes the full contraction of $\mathcal{A}$ with $r$ copies of $\underline{x}$.
The polar glyphs of the computed Minkowski tensors of the centered model Voronoi cell are drawn in Figs.~\ref{fig.3100} to~\ref{fig.3130} for the 0th, 1st and 2nd kind, with orders up to nine.

\begin{figure}
\centering
\subfloat[Nucleation point, cell boundary, mass center and even-ordered Minkowski tensors of 0-th kind.]{\includegraphics[clip,scale=.6]{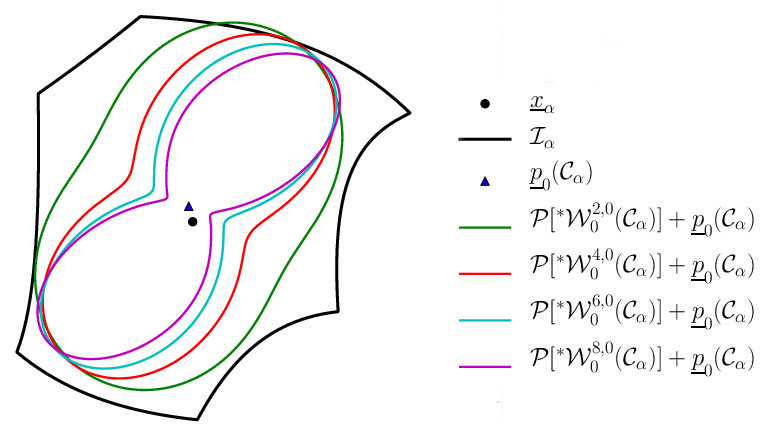}}\hspace{.5cm}
\subfloat[Nucleation point, cell boundary, mass center and odd-ordered Minkowski tensors of 0-th kind.]{\includegraphics[clip,scale=.6]{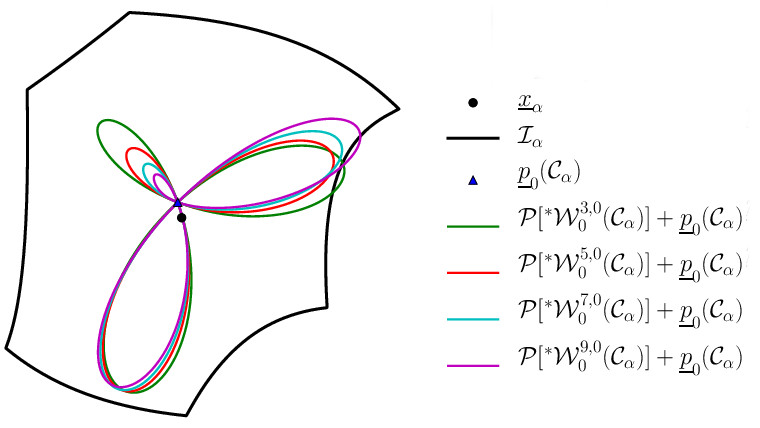}}
\caption{Scaled polar glyphs of normalized Minkowski tensors of 0-th kind for a planar cell $\mathcal{C}_\alpha$ radially convex at $\underline{x}_{\alpha}$.}
\label{fig.3100}
\end{figure}
	
\begin{figure}
\centering
\subfloat[Nucleation point, cell boundary, mass center and even-ordered normal Minkowski tensors of 1-st kind.]{\includegraphics[clip,scale=.6]{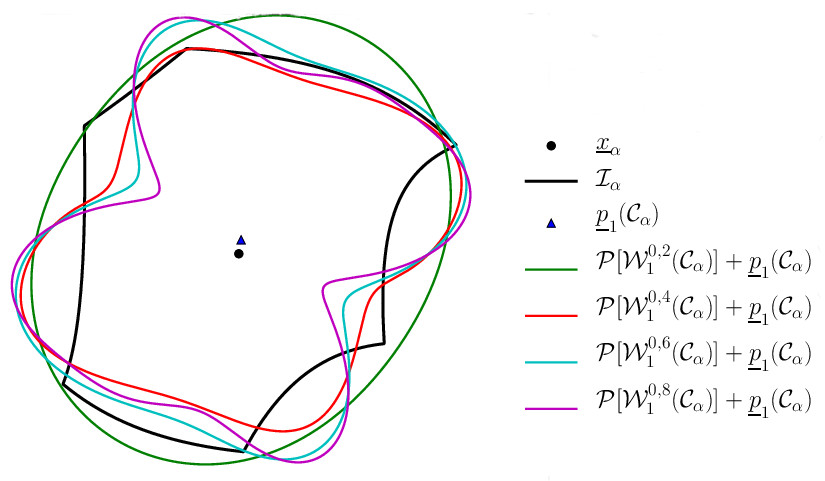}}\hspace{.5cm}
\subfloat[Nucleation point, cell boundary, mass center and odd-ordered normal Minkowski tensors of 1-st kind.]{\includegraphics[clip,scale=.6]{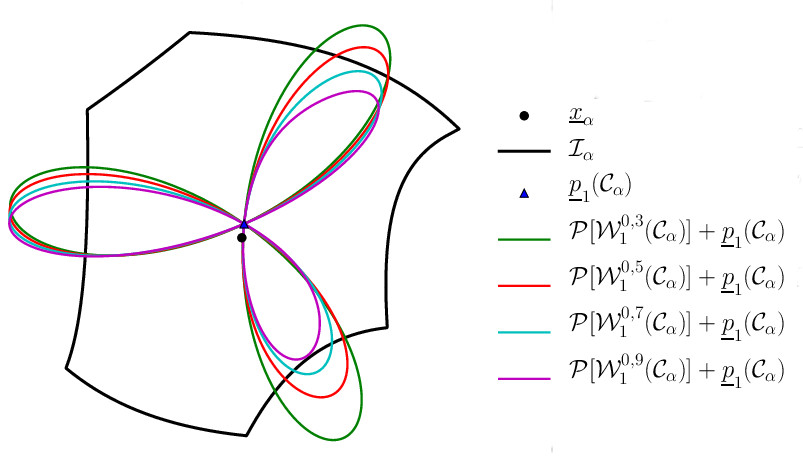}}
\caption{Scaled polar glyphs of normal Minkowski tensors of 1-st kind for a planar cell $\mathcal{C}_\alpha$ radially convex at $\underline{x}_{\alpha}$.}
\label{fig.3110}
\end{figure}
	
\begin{figure}
\centering
\subfloat[Nucleation point, cell boundary, mass center and even-ordered Minkowski tensors of 1-st kind.]{\includegraphics[clip,scale=.6]{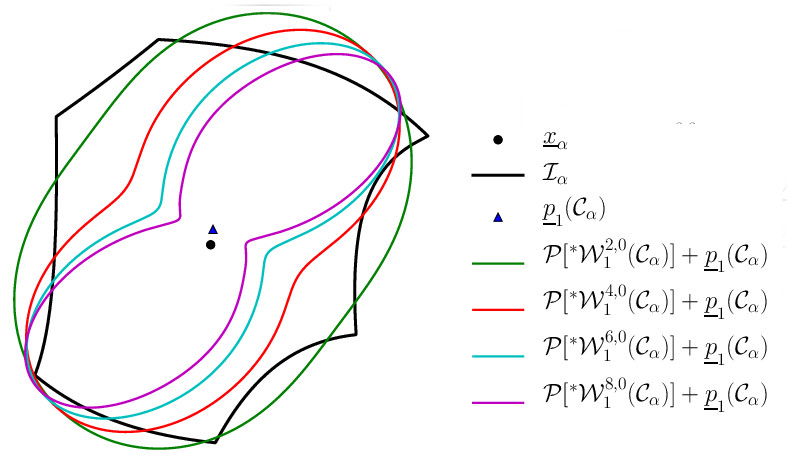}}\hspace{.5cm}
\subfloat[Nucleation point, cell boundary, mass center and odd-ordered Minkowski tensors of 1-st kind.]{\includegraphics[clip,scale=.6]{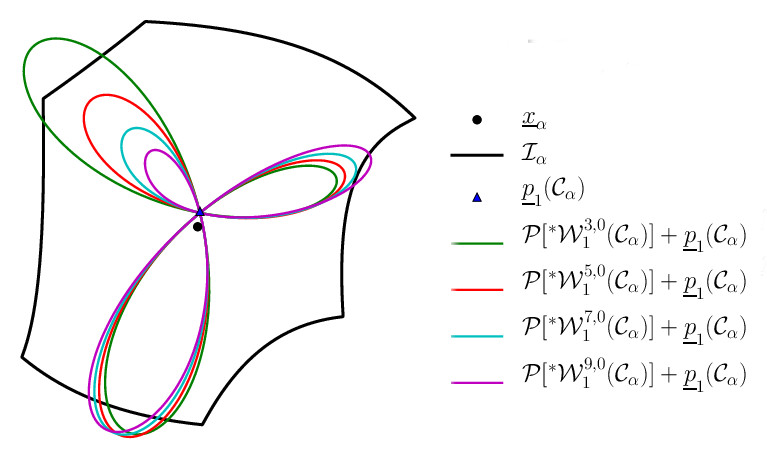}}
\caption{Scaled polar glyphs of normalized Minkowski tensors of 1-st kind for a planar cell $\mathcal{C}_\alpha$ radially convex at $\underline{x}_{\alpha}$.}
\label{fig.3120}
\end{figure}
	
\begin{figure}
\centering
\subfloat[Nucleation point, cell boundary, mass center and even-ordered Minkowski tensors of 2-nd kind.]{\includegraphics[clip,scale=.6]{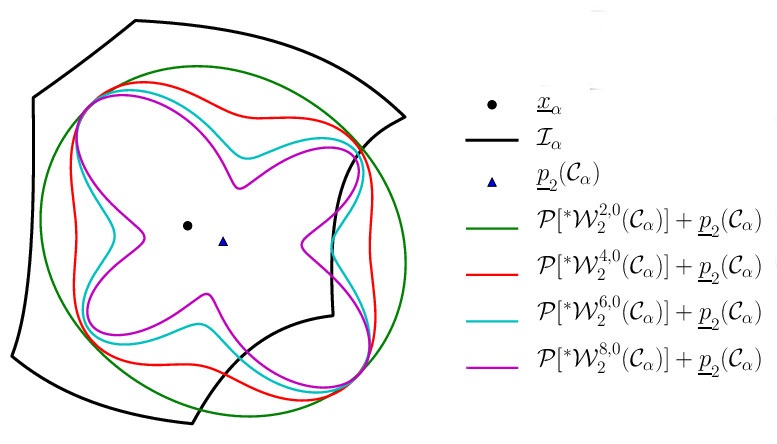}}\hspace{.5cm}
\subfloat[Nucleation point, cell boundary, mass center and odd-ordered Minkowski tensors of 2-nd kind.]{\includegraphics[clip,scale=.6]{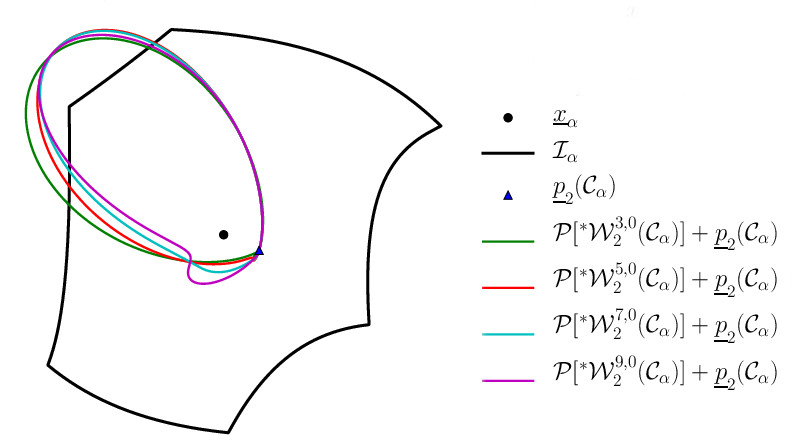}}
\caption{Scaled polar glyphs of normalized Minkowski tensors of 2-nd kind for a planar cell $\mathcal{C}_\alpha$ radially convex at $\underline{x}_{\alpha}$.}
\label{fig.3130}
\end{figure}
	
\section{Conclusion}\label{conclusion}

In this work, we derived semi-analytical formulas for Minkowski tensors of arbitrary order of the 0th, 1st and 2nd kind for planar anisotropic Voronoi cells that are simply connected and star convex at their nucleation point. 
The derivation relies on a diffeomorphic transformation under which every visible segment of the cell boundary is described by a contact function of a single angular variable. 
In terms of this function and its first two derivatives, we expressed the normal and tangent vectors, the surface and length elements, and the curvature of the cell boundary in closed form.
Building on these descriptors, each Minkowski tensor was written as a finite sum of symmetric tensor products of the site position and of the principal directions of the site matrix, weighted by scalar one-dimensional integrals of the contact function. 
The contributions of the common points of the cell boundary to the tensors of the 2nd kind were obtained in closed form through a parallel body construction.

A practical consequence of this formulation is that the tensor structure is separated from the geometry of the cell. 
A single set of scalar integrals, which can be approximated by standard quadrature to arbitrary accuracy, serves all tensor orders, and no discretization of the cell into polygons or voxels is needed. 
The formulas were illustrated on a model cell by means of polar glyphs of tensors up to order nine.

The present results are subject to several limitations, which also point to directions for future work. 
First, anisotropic Voronoi cells are not star convex in general, and may even be disconnected.
Extending the formulas to such cells requires a decomposition of the boundary into segments that are visible from suitable reference points. 
Second, the formulas assume that the contact function is known along the entire cell boundary, which requires the neighbors of the cell and the angular ranges of its boundary segments to be determined beforehand. 
Third, the extension of the approach to three dimensions, where cells are bounded by patches of quadric surfaces and Minkowski tensors of the 3rd kind arise, remains to be carried out. 
Finally, since the approach only relies on the existence of a contact function, we expect it to apply to other tessellations with curved cell boundaries, such as generalized balanced power diagrams, and to enable accurate morphological statistics of anisotropic polycrystalline microstructures.

\begin{appendices}

\section{Binomial theorem for the tensor power of a sum of two vectors}\label{Appendix}
The tensor power of a sum of two vectors $\underline{a}$ and $\underline{b}$ is given by
\begin{equation}
(\underline{a}+\underline{b})^{\otimes^r}=
\sum_{k=0}^{r}\binom{r}{k}\;\underline{a}^{\otimes^{r-k}}\!\!\odot\underline{b}^{\otimes^k}
\end{equation}
where $\odot$ denotes the symmetric tensor product.

\end{appendices}

\bibliography{sn-bibliography.bib}

\end{document}